\documentclass[11pt]{amsart}
\usepackage{amsmath,amssymb,amsthm}
\usepackage[latin1]{inputenc}
\usepackage{tabularx,multicol,array}
\usepackage{graphicx,float,psfrag}
 \usepackage{stmaryrd,mathrsfs}
\usepackage{url}
\usepackage{bbm}
\usepackage{tikz}
\usepackage{caption}
\usepackage{subcaption}
\usepackage{psfrag}

\newcommand{\esp}{\vspace{.2cm}}

\newcommand{\reff}[1]{(\ref{#1})}

\theoremstyle{plain}
\newtheorem{theo}{Theorem}[section]
\newtheorem{theo*}{Theorem}
\newtheorem{cor}[theo]{Corollary}
\newtheorem{prop}[theo]{Proposition}
\newtheorem{lem}[theo]{Lemma}

\theoremstyle{remark}
\newtheorem{rem}[theo]{Remark}

\newcommand{\ca}{{\mathcal A}}

\newcommand{\cc}{{\mathcal C}}
\newcommand{\cd}{{\mathcal D}}

\newcommand{\cg}{{\mathcal G}}

\newcommand{\ci}{{\mathcal I}}
\newcommand{\cj}{{\mathcal J}}

\newcommand{\N}{{\mathbb N}}
\renewcommand{\P}{{\mathbb P}}

\newcommand{\R}{{\mathbb R}}

\newcommand{\ind}{{\bf 1}}

\newcommand{\val}[1]{\mathop{\left| #1 \right|}\nolimits}
\newcommand{\inv}[1]{\mathop{\frac{1}{ #1}}\nolimits}
\newcommand{\expp}[1]{\mathop {\mathrm{e}^{ #1}}}

\renewcommand{\phi}{\varphi}
\renewcommand{\epsilon}{\varepsilon}

\DeclareMathOperator\arctanh{arctanh}

\title{Maximum entropy copula with given diagonal section}
\date{\today}

 \author{Cristina Butucea}
 \address{
 Cristina Butucea,
 Université Paris-Est, LAMA (UPE-MLV), F-77455 Marne La Vallée, France.}
 \email{cristina.butucea@univ-mlv.fr}

 \author{Jean-François Delmas}
 \address{
 Jean-Fran\c cois Delmas,
 Université Paris-Est, CERMICS (ENPC), F-77455 Marne La Vallée, France.}
 \email{delmas@cermics.enpc.fr}

 \author{Anne Dutfoy}
 \address{
 Anne Dutfoy, 
 EDF Research \& Development, Industrial Risk Management Department, 92141 Clamart Cedex, France.}
 \email{anne.dutfoy@edf.fr}

 \author{Richard Fischer}
 \address{
 Richard Fischer, 
 Université Paris-Est, CERMICS (ENPC), F-77455 Marne La Vallée, France\\
 EDF Research \& Development, Industrial Risk Management Department, 92141 Clamart Cedex, France.}
 \email{fischerr@cermics.enpc.fr}

\begin{document}

\thanks{This work is partially supported by the French ``Agence Nationale de
 la Recherche'',CIFRE n° 1531/2012, and by EDF Research \& Development, Industrial Risk Management Department}

\keywords{copula, entropy, diagonal section}

\subjclass[2010]{62H05,60E05}

 \begin{abstract} 
 We consider copulas with a given diagonal section and compute the explicit density of the unique optimal 
 copula which maximizes the entropy. In this sense, this copula 
 is the least informative among the copulas with a given diagonal section. We give an explicit criterion 
 on the diagonal section for the existence of the optimal copula and give a closed formula for its
 entropy. We also provide examples for some diagonal sections of usual bivariate copulas and illustrate the
 differences between them and the maximum entropy copula with the same diagonal section.
\end{abstract}

\maketitle

\section{Introduction}

Dependence of random variables can be described by copula distributions. A copula is the cumulative distribution 
function of a random vector $U=\left(U_1, \hdots, U_d\right)$ with $U_i$ uniformly distributed on 
$I=[0,1]$. For an exhaustive overview on copulas, we refer to {\sc Nelsen} \cite{Nelsen1999}. The diagonal section $\delta$
of a $d$-dimesional copula $C$, defined on $I$ as $\delta(t)=C(t,\hdots,t)$ is the cumulative distribution function of 
$\max_{1 \leq i \leq d} U_i$. The function $\delta$ is non-decreasing, $d$-Lipschitz, and verifies $ \delta(t) \leq t$ for all
$t \in I$ with $\delta(0)=0$ and $\delta(1)=1$. It was shown that if a function $\delta$ satisfies these properties, then 
 there exists a copula with $\delta$ as diagonal section (see {\sc Bertino} \cite{bertino} or {\sc Fredricks and Nelsen} 
 \cite{Fredricks1997} for $d=2$ and {\sc Cuculescu and Theodorescu} \cite{cuculescu2001copulas} for $d \geq 2$ ).

 Copulas with a given diagonal section have been studied in different papers, as the diagonal sections  are 
 considered in various fields of application. Beyond the fact that $\delta$ is the cumulative distribution function of the 
 maximum of the marginals, it also characterizes the tail dependence of the copula (see {\sc Joe}
 \cite{Joe1997} p.33. and references  in {\sc Nelsen et al.}  \cite{Nelsen2008473}, {\sc Durante and Jaworski} \cite{durante2008absolutely}, {\sc Jaworski} \cite{Jaworski20092863}) 
 as well as the generator for Archimedean copulas ({\sc Sungur and Yang} \cite{SungurYang1996}). For $d=2$, Bertino in \cite{bertino}
 introduces the so-called Bertino copula $B_{\delta}$ given by $B_{\delta}(u,v)=u \wedge v-\min_{u \wedge v
 \leq t \leq u \vee v}(t-\delta(t))$ for $u,v \in I$. Fredricks and Nelsen in \cite{Fredricks1997} give the example called diagonal 
 copula defined by $K_{\delta}(u,v)=\min(u,v,(\delta(u)+\delta(v))/2)$ for $u,v \in I$. In {\sc Nelsen et al.} \cite{Nelsen2004348,Nelsen2008473}
 lower and upper bounds related to the pointwise partial ordering are given for copulas with a given diagonal section. They
 showed that if $C$ is a symmetric copula with diagonal section $\delta$, then for every $u,v \in I$, we have:
 $$B_{\delta}(u,v) \leq C(u,v) \leq K_\delta(u,v).$$
 
 {\sc Durante et al.} \cite{durantemesiar} provide another construction of copulas for a certain class of diagonal sections, called MT-copulas named after Mayor and 
 Torrens and defined as $D_{\delta}(u,v)=\max(0,\delta(x \vee y) - |x-y|)$. Bivariate copulas with given sub-diagonal sections 
 $\delta_{x_0} : [0,1-x_0] \rightarrow [0,1-x_0], \delta_{x_0}(t)=C(x_0+t,t)$ are 
 constructed from copulas with given diagonal sections in {\sc Quesada-Molina et al.} \cite{QuesadaMolina20084654}. {\sc Durante et al.} \cite{Durante2007}
 or \cite{Nelsen2008473} introduce the technique of diagonal splicing to create new copulas with a given diagonal section based on other such copulas. 
 According to \cite{durante2008absolutely} for $d=2$ and {\sc Jaworski}  \cite{Jaworski20092863} for $d \geq 2$, 
 there exists an absolutely continuous copula with diagonal section
 $\delta$ if and only if the set $\Sigma_{\delta}= \{ t \in I; \delta(t)=t \}$ has zero Lebesgue measure.
 {\sc de Amo et al.} \cite{deAmo2012} is an extension of \cite{durante2008absolutely} for given sub-diagonal sections. 
 Further construction of possibly asymmetric absolutely continuous bidimensional copulas with a given diagonal section 
 is provided in {\sc Erdely and Gonz\'{a}lez} \cite{ErdelyGonzalez2006}. 
 
 Our aim is to find the most uninformative copula with a given diagonal section $\delta$.
 We choose here to maximize the relative entropy to the uniform distribution on $I^d$, among the copulas with given diagonal section.
 This is equivalent to minimizing the Kullback-Leibler divergence with respect to the independent copula. The Kullback-Leibler
 divergence is finite only for absolutely continuous copulas. The previously introduced bivariate 
 copulas $B_\delta$, $K_\delta$ and
 $D_\delta$ are not absolutely continuous, therefore their Kullback-Leibler divergence is infinite.  
 Possible other entropy criteria, such as R\'{e}nyi, Tsallis, etc. are considered for example in {\sc Pougaza and Mohammad-Djafari} \cite{pougaza2010link}.
 We recall that the entropy of a $d$-dimensional absolutely continuous random vector $X=(X_1,\hdots, X_d) $ can be 
 decomposed as the sum of the entropy of the marginals 
 and the entropy of the corresponding copula (see {\sc Zhao and Lin} \cite{Zhao2011628}) :
 $$H(X)=\sum_{i=1}^dH(X_i)+H(U),$$
 where $H(Z)=-\int f_Z(z) \log f_Z(z) dz $ is the entropy of the random variable $Z$ with density $f_Z$,
 and $U=(U_1, \hdots, U_d)$ is a random vector with $U_i$ uniformly distributed on $I$,
 such that $U$ has the same copula as $X$; namely $U$ is distributed as $\left(F_1^{-1}(X_1), \hdots F_d^{-1} (X_d)\right)$
 with $F_i$ the cumulative distribution function of $X_i$.
 Maximizing the entropy of $X$ with given marginals therefore corresponds
 to maximizing the entropy of its copula.
 The maximum relative entropy approach for copulas has an extensive litterature.  Existence results for an optimal solution on
 convex closed subsets of copulas for the total variation distance can be derived from { \sc Csisz\'{a}r } \cite{csiszar1975divergence}. A general discussion
 on abstract entropy maximization is given by {\sc Borwein et al.} \cite{borwein1994entropy}. This theory was applied for copulas 
 and a finite number of expectation constraints in {\sc Bedford and Wilson} \cite{Bedford2013}. Some applications for 
 various moment-based constraints include rank correlation ({\sc Meeuwissen and Bedford} \cite{meeuwissen1997minimally},
 {\sc Chu} \cite{chu2011recovering}, {\sc Piantadosi et al.} \cite{piantadosi}) and marginal moments ({\sc Pasha and Mansoury} \cite{pasha2008determination}). 
 
 We shall apply the theory developed in \cite{borwein1994entropy} to compute the density of the maximum entropy
 copula with a given diagonal section.  We show that there exists a copula with diagonal section $\delta$ and
 finite entropy if and only if $\delta$ satisfies: $ \int_I |\log(t-\delta(t))| dt < + \infty$. Notice
 that this condition is stronger than the condition of $\Sigma_{\delta}$ having zero Lebesgue measure 
 which is required for the existence of an absolutely continuous copula with diagonal section $\delta$. Under this condition,
 and in the case of $\Sigma_\delta=\{0,1\}$, 
 the  optimal copula's density  $c_\delta$ turns out to be of the form, for $x=(x_1, \hdots, x_d) \in I^d$:
 \[
 c_\delta(x)=b(\max(x))\prod_{x_i \neq \max(x)} a(x_i), 
 \]
 with the notation $\max(x)=\max_{1 \leq i \leq d} x_i$, see Theorem \ref{theo:spec}.
 The optimal copula's density in the general case is given in Theorem \ref{theo:gen}. Notice that
 $c_\delta$ is symmetric, that is it is invariant under the permutation of the variables. 
 This provides a new family of absolutely continuous symmetric copulas
 with given diagonal section enriching previous work on this subject that we discussed, see 
 \cite{bertino},\cite{durante2008absolutely},\cite{Durante2007},\cite{durantemesiar},\cite{ErdelyGonzalez2006},\cite{Fredricks1997},\cite{Nelsen2008473}.
 We also calculate the maximum entropy copula for diagonal sections that arise from well-known families of bivariate copulas. 
 
 The rest of the paper is organised as follows. Section \ref{sec:main} introduces the definitions and notations used later on, and gives 
 the main theorems of the paper.  In Section \ref{sec:S=01} we study the properties of the feasible
 solution $c_{\delta}$ of the problem for a special class of diagonal sections with $\Sigma_\delta = \{0,1\}$. 
 In Section \ref{sec:min_prob}, we formulate our problem as a linear optimization problem in order to 
 apply the theory established in \cite{borwein1994entropy}.
 Then in Section \ref{sec:proof-theo-spec} we give the proof for our main theorem showing that $c_{\delta}$ is indeed the optimal solution
 when $\Sigma_\delta = \{0,1\}$.
 In Section \ref{sec:proof-theo-gen} we extend our results for the general case when $\Sigma_\delta$ has zero Lebesgue measure. 
 We give in Section \ref{sec:examples} several examples with diagonals of popular bivariate 
 copula families such as the Gaussian, Gumbel or Farlie-Gumbel-Morgenstern copulas among others. 
 
\section{Main results} \label{sec:main}
 Let $d \geq 2$ be fixed.
We  recall  a  function $C$  defined  on  $I^d$,  with $I=[0,1]$,  is  a
$d$-dimensional copula if there exists a random vector $U=(U_1, \hdots, U_d)$
such that $U_i$ are uniform  on $I$ and  $C(u)=\P(U \leq u)$ for
$u\in I^d$, with the convention that $x \leq y$ for $x=(x_1,\hdots.x_d)$ and 
$y=(y_1,\hdots, y_d)$ elements of $\R^d$ if and only if $x_i \leq y_i$ for all
$1\leq i \leq d$. 
We shall say that $C$ is the copula of $U$. We refer to
\cite{Nelsen1999} for  a monograph on  copulas. The copula $C$  is said
absolutely  continuous if  the random  variable $U$  has  a density,
which  we   shall  denote  by  $c_C$.   In  this  case,   we  have  that
a.e. $c_C(u) = \partial^d_{u_1,\hdots,u_d} \,  C(u)$ for $u\in I^d$. When there
is no confusion, we shall write  $c$ for the density $c_C$ associated to
the copula $C$.   We denote by $\cc$ the  set of $d$-dimensional copulas
and by $\cc_0$ the
subset of  the $d$-dimensional absolutely continuous  copulas.  

\esp The diagonal section $\delta_C$ of a copula
$C$  is defined  by: $\delta_C(t)=C(t,\hdots,t)$. Let us note, for $u \in \R^d$,
$\max(u)=\max_{1\leq i \leq d} u_i$.  Notice  that if  $C$ is  the
copula  of $U$, then  $\delta_C$ is  the cumulative distribution function of 
$\max(U)$ as $\delta_C(t)=\P(\max(U)\leq t)$ for $t\in I$.
We denote by $\cd=\{\delta_C, C\in \cc\}$ the set of diagonal sections of
$d$-dimensional copulas and by $\cd_0=\{\delta_C; C\in  \cc_0\}$ the set of
diagonal sections of absolutely continuous copulas. 
According to \cite{Fredricks1997}, a function
$\delta$ defined on $I$ belongs to $\cd$ if and only if:
\begin{itemize}
\item[(i)]  $\delta$ is  a  cumulative   function  on $[0,1]$:  $
  \delta(0)=0$, $\delta(1)=1$ and $ \delta $ is non decreasing;
 \item[(ii)]  $ \delta(t) \leq t$ for $t \in I$ and $\delta$ is  $d$-Lipschitz: $
   \val{\delta(s)-\delta(t)} \leq d\val{s-t} $ for $s,t\in I$. 
\end{itemize}
For $\delta\in  \cd$, we shall consider the  set $\cc^\delta=\{C\in \cc;
\delta_C=\delta\}$ of  copulas with  diagonal section $\delta$,  and the
subset  $\cc^\delta_0=\cc^\delta\bigcap \cc_0$ of  absolutely continuous
copulas       with      section       $\delta$.        According      to
\cite{durante2008absolutely}   and   \cite{Jaworski20092863},  the   set
$\cc^\delta_0$    is   non   empty    if   and    only   if    the   set
$\Sigma_{\delta}=\{t\in I; \delta(t)=t\}$  has zero Lebesgue measure. 

\esp For a non-negative measurable function $f$ defined on $I^k,k \in \N^*$, we set
\[
\ci_k(f) = \int_{I^k} f(x)\log(f(x)) \, dx.
\]
Since copulas  are cumulative
functions of probability measures, we will
consider  the  Kull\-back-Leibler   divergence  relative  to  the  uniform
distribution as a measure of entropy, see \cite{csiszar1975divergence}:
\[
\ci(C)=
\begin{cases}
   \ci_d(c) & \text{ if $C\in
     \cc_0$,}\\
   +\infty  & \text{ if $C\not \in
     \cc_0$,}
\end{cases}
\]
with $c$ the density associated to $C$ when $C\in\cc_0$. 
Notice  the  Shannon-entropy  introduced  in \cite{Shannon1948}  of  the
probability  measure $P$  defined  on $I^d$  with cumulative  distribution
function $C$ is defined  as $H(P)=-\ci(C)$. Thus minimizing the Kullback-Leibler
divergence $\ci$  (w.r.t.  the  uniform   distribution)  is   equivalent  to
maximizing the Shannon-entropy. It is well known that the copula $\Pi$
with density $c_\Pi=1$, which corresponds to $(U_i, 0 \leq i \leq d)$ being independent,
minimizes $\ci(C)$ over $\cc$.

\esp  We shall minimize the entropy $\ci$ over the set $\cc^\delta$  or  equivalently over
$\cc^\delta_0$ of copulas with a given diagonal section $\delta\in \cd$ (in fact for
$\delta\in \cd_0$ as otherwise $\cc^\delta_0$  is empty). If
$C$  minimizes $\ci$  on $\cc^\delta$,  it means  that $C$  is  the least
informative (or the ``most  random'') copula with given diagonal section
$\delta$.

\esp For $\delta\in \cd$, let us denote:
\begin{equation}
 \label{eq:def-jd}
\cj(\delta)= \int_I \val{\log (t-\delta(t))}\, dt.
\end{equation}
Notice that $\cj(\delta)\in [0,+\infty ]$  and it is infinite if $\delta\not\in
\cd_0$. Since $\delta$  is  $d$-Lipschitz, the  derivative $\delta'$  of
$\delta$ exists a.e. and since $\delta$  is non-decreasing we have a.e. $0\leq
\delta'\leq d$. This implies that $\ci_1(\delta')$ and $\ci_1(d-\delta')$ are well
defined. Let us denote:
\begin{equation}
 \label{eq:def-gd}
\cg(\delta)=\ci_1(\delta')+\ci_1(d-\delta')- d \log(d) -(d-1).
\end{equation}
 We have the rough upper bound:
\begin{equation}
   \label{eq:D-bound}
\sup_{\delta\in \cd} |\cg(\delta)|\leq  d+d\log(d).
 \end{equation}

The following Proposition gives an absolutely continuous copula whose diagonal section is 
$\delta$. The proof of this Proposition can be found in Section \ref{sec:S=01} and 
Section \ref{app:proofKL} is dedicated to the proof of \reff{eq:ci=cj+g}.
\begin{prop}
 \label{lem:cd=density} 
 Let $\delta  \in \cd_0  $ with  $\Sigma_{\delta}=\{0,1\}$. We define, for $r\in I$:
\begin{equation*}
  h(r)=r-\delta(r), \quad
  F(r)= \frac{d-1}{d}\int_{\inv{2}}^r \frac{1}{h(s)} \, ds,
\end{equation*}

\begin{equation}
   \label{eq:a-b}
    a(r) = \frac{d- \delta'(r)}{d} h(r)^{-1+1/d}\expp{F(r)} \quad \text{and} \quad
    b(r) = \frac{ \delta'(r)}{d}h(r)^{-1+1/d}\expp{-(d-1)F(r)}. 
\end{equation} 
Then $c_\delta$ defined a.e. by
\begin{equation}
\label{eq:cd_def}
 c_{\delta}(x) = b(\max(x)) \prod_{x_i \neq \max(x)} a(x_i) \,  \quad \quad \quad 
 \text{ for }\,  x=(x_1, \hdots, x_d)\in I^d,
\end{equation}
is the density of a symmetric copula $C_\delta$ with diagonal section $\delta$.  
Furthermore, we have:
\begin{equation}
   \label{eq:ci=cj+g}
\ci(C_\delta)= (d-1)\cj(\delta) + \cg(\delta).
\end{equation}
\end{prop}
This and \eqref{eq:D-bound} readily
implies the following Remark.
\begin{rem}  
\label{lem:finiteKL}  
Let $\delta  \in  \cd_0  $ such  that
  $\Sigma_{\delta}=\{0,1\}$.  We have $\ci(C_{\delta})  < + \infty$ if and
  only if $\cj(\delta)  < +\infty$. 
\end{rem}
We can now state our main result in the simpler case
$\Sigma_{\delta}=\{0,1\}$. It gives the necessary and
sufficient condition for $C_\delta$ to be the 
unique optimal solution of the minimization problem.
The proof is given in Section
\ref{sec:proof-theo-spec}. 
\begin{theo} 
\label{theo:spec}
Let $\delta \in \cd_0 $ such that $\Sigma_{\delta}=\{0,1\}$. 
\begin{itemize}
   \item[a)] If $ \cj(\delta)= +\infty$ then
     $\min_{C \in \cc^\delta} \ci(C)=+\infty $. 
   \item[b)] If $\cj(\delta) < +\infty$ then
$ \min_{C \in \cc^\delta} \ci(C)<+\infty $ and  $C_\delta$ is the unique copula such that
$\ci\left(C_\delta\right)=\min_{C \in \cc^\delta}\ci(C)$.
\end{itemize}
\end{theo}

\esp To give  the answer in the general case where $\Sigma_\delta$ has zero Lebesgue measure, we need  some extra notations.
Since $\delta$ is continuous,  we get that $I \setminus \Sigma_{\delta}$
can be  written as  the union of  non-empty open  intervals $((\alpha_j,
\beta_j),   j\in  J)$,   with   $\alpha_j<\beta_j$  and   $J$  at   most
countable.      Notice     that      $\delta(\alpha_j)=\alpha_j$     and
$\delta(\beta_j)=\beta_j$. For $J\neq \emptyset$ and $j\in J$, we set
$\Delta_j=\beta_j -\alpha_j$ and for $t\in I$:
\begin{equation}
   \label{eq:delta-j}
 \delta^j(t) =
 \frac{\delta\left(\alpha_j+ t\Delta_j \right)-\alpha_j}{\Delta_j} \cdot
\end{equation}
It is clear that $\delta^j$ satisfies (i) and (ii) and it belongs to
$\cd_0$ as  $\Sigma_{\delta^j}=\{0,1\}$. 
Let $c_{\delta^j}$ be defined by \reff{eq:cd_def} with $\delta$ replaced
by $\delta^j$. For $\delta\in \cd_0$ such that $\Sigma_\delta\neq \{0,1\}$, we define
the function  $c_\delta$ by, for $u\in I^d$:
\begin{equation}
   \label{eq:density-gen}
c_\delta(u)=\sum_{j\in J} \inv{\Delta_j} c_{\delta^j} 
\left(\frac{u-\alpha_j\mathbf{1}}{\Delta_j} \right) \; 
\,\ind_{(\alpha_j,\beta_j)^d}(u),
\end{equation}
with $\mathbf{1}=(1,\hdots,1) \in \R^d$.
It is easy to check that $c_\delta$ is a copula density and that is zero
outside $[\alpha_j,\beta_j]^d$ for $j\in J$.   We state our
main  result  in  the general  case  whose  proof  is given  in  Section
\ref{sec:proof-theo-gen}.

\begin{theo} 
\label{theo:gen}
Let $\delta \in \cd $. 
\begin{itemize}
   \item[a)] If $\cj(\delta) = +\infty$ then
     $\min_{C \in \cc^\delta} \ci(C)=+\infty $. 
   \item[b)] If $\cj(\delta)  < +\infty$ then
$ \min_{C \in \cc^\delta} \ci(C)<+\infty $ and there exists a unique
copula $C_\delta\in \cc^\delta$ such that
$\ci\left(C_\delta\right)=\min_{C \in \cc^\delta}\ci(C)$. 
Furthermore, we have:
\[
\ci(C_\delta)= (d-1)\cj(\delta) + \cg(\delta);
\]  
the copula $C_\delta$ is absolutely continuous,
symmetric;  its density $c_\delta$ is given by 
\reff{eq:cd_def} if $\Sigma_\delta= \{0,1\}$ or by 
\reff{eq:density-gen} if $\Sigma_\delta\neq \{0,1\}$.
\end{itemize}
\end{theo}

\begin{rem}
   \label{rem:abs-cont}
For $\delta\in \cd$, notice the condition $\cj(\delta)<+\infty $ implies
that $\Sigma_\delta$ has zero Lebesgue measure, and therefore, according
to \cite{durante2008absolutely}  and \cite{Jaworski20092863}, $\delta\in
\cd_0$.  And  if $\delta\not\in \cd_0$,  then $\ci(C)=+\infty $  for all
$C\in \cc^\delta$. Therefore, we could replace the condition $\delta \in
\cd $ by $\delta \in \cd_0 $ in Theorem \ref{theo:gen}.
\end{rem}

\section{Proof of Proposition \ref{lem:cd=density} }
\label{sec:S=01}
We assume that  $\delta \in \cd_0 $ and  $\Sigma_{\delta}=\{0,1\}$. 
We give the proof of Proposition \ref{lem:cd=density}, which states that 
$C_\delta$, with density $c_\delta$ given by \reff{eq:cd_def}, is indeed 
a symmetric copula with diagonal section $\delta$ whose entropy is given by
\reff{eq:ci=cj+g}.

Recall the definition of $h,F,a,b$ and $c_\delta$ from Theorem \ref{theo:spec}.
Notice that by construction $c_\delta$ is non-negative and well defined
on $I^d$. In order to prove that $c_\delta$ is the density of a copula, we only have to
prove that for all $1 \leq i \leq d$, $r\in I$:
 \[
\int_{I^d} c_\delta(u) \ind_{\{u_i\leq r\}} \, du  = r,
\]
or equivalently
 \[
\int_{I^d} c_\delta(u) \ind_{\{u_i\geq r\}} \, du  = 1-r.
\]
We define for $r\in I$:
\begin{equation}
   \label{eq:A-B-int}
A(r)=\int_0^r a(t) \,dt .
\end{equation}
Elementary computations yield for $r\in (0,1)$:
\begin{equation}
   \label{eq:A-B}
  A(r)=h^{1/d}(r)\,\expp{ F(r)} .
\end{equation}
Notice that $F(0)\in [-\infty ,0]$  which
implies that $A(0)=0$. 
A direct integration gives:
\begin{equation} 
\label{eq:A^(d-1)}
 d\int_I A^{d-1}(s) b(s)\ind_{\{s \geq r\}}=1-\delta(r).
\end{equation}
We also have: 
\begin{align}
 \label{eq:A^(d-2)}
 \nonumber (d-1)\int_I A^{d-2}(s) b(s) \, ds\ind_{\{s \geq r\}} 
           & = \frac{(d-1)}{d}\int_I h^{-1/d}(s) \expp{-F(s)}\ind_{\{s \geq r\}} \, ds\\
 \nonumber & = \left[-h^{1-1/d} (s)\expp{-F(s)}\right]^1_{s=r} \\ 
           & = h^{1-1/d} (r)\expp{-F(r)},
\end{align}
where we used for the last step that $h(1)=0$ and $F(1)\in [0,\infty]$. We have:
\begin{align*}
  \int_{I^d} c_\delta(u) \ind_{\{u_i \geq r\}} \, du
&=  \int_{I^d} b(\max(u))\prod_{u_j \neq \max(u)} a(u_j)   \ind_{\{u_i \geq r\}} \, du  \\
&=  \int_I A^{d-1}(s) b(s) \ind_{\{s \geq r\}} \, ds \\
&   \hspace{1.5cm}  + (d-1)\int_I A^{d-2}(s) b(s) (A(s)-A(r)) \ind_{\{s \geq r\}} \, ds \\
&=  d\int_I A^{d-1}(s) b(s) \ind_{\{s \geq r\}} \, ds \\
&   \hspace{1.5cm}  - (d-1)A(r)\int_I A^{d-2}(s) b(s) \ind_{\{s \geq r\}} \, ds \\
&= 1-\delta(r) - (r-\delta(r)) \\
&= 1-r,
\end{align*}
where we first divided the integral according to which $u_i$ was the maximum; then we 
used \reff{eq:A-B-int} for the second equality, finally \reff{eq:A^(d-1)} and \reff{eq:A^(d-2)}
for the forth.
This implies that $c_\delta$ is indeed the density of a copula.
We denote by $C_\delta$ the copula with density $c_\delta$. We check that
$\delta$ is the diagonal section of $C_\delta$. 
Using \reff{eq:A^(d-1)}, we get, for $r\in I$:
\begin{align*}
  \int_{I^d} c_\delta(u) \ind_{\{\max(u) \leq  r\}} \, du 
  &    =  \int_{I^d} b(\max(u)) \prod_{u_i\neq \max(u_i)} a(u_i) \ind_{\{\max(u) \geq r\}} \, du \\
  &    = d\int_{I} A^{d-1}(s)b(s)  \ind_{\{s\leq r\}} \, ds \\
  &    = \delta(r). 
\end{align*}

The calculations which show that the entropy of
$C_\delta$ is given by \reff{eq:ci=cj+g} 
can be found in \mbox{Section \ref{app:proofKL}}.



\section{The minimization problem} \label{sec:min_prob}

Let $\delta \in \cd_0$.
As a first step we  will show, using \cite{borwein1994entropy}, that the
problem of a maximum entropy copula with a given diagonal section
$\delta$  has at  most a  unique  optimal solution.   To formulate  this
problem in the framework  of \cite{borwein1994entropy}, we introduce the
continuous  linear  functional   $\ca=(\ca_i, 1 \leq i \leq d+1)  :L^1(I^d)
\rightarrow L^1(I)^{d+1}$ defined by, for $1 \leq i \leq d$, $f\in L^1(I^d)$ and $r\in I$,
\[
\ca_i(f)(r)=\int_{I^d} f(u)\ind_{\{u_i\leq r\}} \, du, \quad
 \text{and} \quad
\ca_{d+1}(f)(r)=\int_{I^d} f(u)\ind_{\{\max(u)\leq r\}} \, du.
\]
We also  define $b^\delta=(b_i, 1 \leq i \leq d+1) \in  L^1(I)^{d+1}$ with $b_{d+1}=\delta$
and $b_i=\text{id}_I$ for $1 \leq i \leq d$, with $\text{id}_I $ the identity map on $I$.  Notice that the
conditions  $\ca_i(c)=b_i$, $1 \leq i \leq d$,  and $c\geq  0$  a.e.  imply
that $c$ is  the density of a copula $C\in \cc_0$.  If we assume further
that the condition $\ca_{d+1}(c)=b_{d+1}$ holds then the diagonal section of $C$
is $\delta$ (thus $C\in \cc^\delta_0$).

Since $\ci$ is infinite outside $\cc_0^\delta$  and the density of
any copula in  $\cc_0$ belongs to $L^1(I^d)$, we get that  minimizing
$\ci$ over $\cc^\delta$ is  equivalent to
the  linear optimization problem $(P^\delta)$ given by:
\begin{equation}
\tag{$P^{\delta}$}
\text{minimize } \ci_d(c) \text{ subject to } 
\begin{cases}
 &  \ca(c)=b^\delta,\\
&c\geq 0 \text{ a.e. and } c\in L^1(I^d). 
\end{cases}
\end{equation}
We  say that  a function  $f$ is  feasible for  $(P^{\delta})$  if $f\in
L^1(I^d)$, $f\geq  0$ a.e.,  $\ca(f)=b^\delta$ and $\ci_d(f)<+\infty  $. Notice
that any feasible $f$ is the density of a copula. We say that $f$ is an
optimal solution  to  $(P^{\delta})$ if $f$ is feasible and
$\ci_d(f)\leq  \ci_d(g)$ for all $g$ feasible.

\begin{prop}
   \label{prop:sym}
   Let  $\delta\in \cd$.  If there  exists  a feasible  $c$, then  there
   exists  a  unique  optimal  solution  to  $(P^{\delta})$  and  it  is
   symmetric.
\end{prop}

\begin{proof}
Since $\ca(f)=b^\delta$ implies $\ca_1(f)(1)=b_1(1)$ that is $ \int_{I^d}f(x)\,
dx=1$,  we can directly apply Corollary~2.3 of \cite{borwein1994entropy} which
states that if there exists a feasible $c$, then there exists a
unique optimal solution to $(P^{\delta})$. Since the constraints are symmetric
and the functional $\ci_d$ is also symmetric, we deduce that the unique optimal
solution is also symmetric.
\end{proof}

The next Proposition gives that the set of  zeros of any
non-negative solution $c$ of $\ca(c)=b^\delta$ contains:
\begin{equation}
   \label{eq:def-Z}
Z_\delta=\{u \in I^d; \delta'(\max(u))=0 \text{ or } \exists i \text{ such that } u_i < \max(u) \text{ and } \delta'(u_i)=d\}. 
\end{equation}

\begin{prop}
   \label{prop:0}
   Let  $\delta\in \cd$.  If   $c$ is feasible then $c=0$ a.e. on
   $Z_\delta$ (that is $c\ind_{Z_\delta}=0$ a.e.). 
\end{prop}
\begin{proof}
Recall that  $0\leq \delta'\leq d$. 
Since $c\in L^1(I^d)$, the condition $\ca_{d+1}(c)=b_{d+1}$, that is for all
$r\in I$
\[
\int_{I^d} c(u) \ind_{\{\max(u)\leq r\}} \, du = \int_0^r
\delta'(s)\, ds, 
\]
 implies, by the monotone class theorem,  that for all
measurable subset $H$ of $I$, we have:
\[
\int_{I^d} c(u) \ind_H(\max(u)) \, du = \int_H
\delta'(s)\, ds. 
\]
Since $c\geq 0$ a.e., we deduce that a.e. $c(u)\ind_{\{\delta'(\max(u))=0\}}=0$. 

Next, notice that for all $r\in I$,
$ 1 \leq i \leq d$, the symmetrical property of $c$ gives:
\begin{align*}
\int_{I^d} c(u) \ind_{\{u_i < \max(u), u_i \leq r\}}\, du & = \int_{I^d} c(u) \ind_{\{u_i \leq r\}}\, du 
-\int_{I^d} c(u) \ind_{\{u_i = \max(u), u_i \leq r\}}\, du \\
                                                             & = r - \frac{\delta(r)}{d} \\                                                    
                                                             & = \int_0^r \left(1-\frac{\delta'(s)}{d}\right) \, ds.
\end{align*}
This implies that a.e. $c(u)\ind_{\{\exists i \text{ such that } u_i < \max(u), \delta'(u_i)=d\}}=0$.
This gives the result. 
\end{proof}

We define $\mu$ to be the Lebesgue measure restricted to $Z_\delta^c=I^d \setminus Z_\delta$:
$\mu(du) = \ind_{Z_\delta^c} (u) du$. We define, for $f\in L^1(I^d, \mu)$:  
\[
 \ci^\mu(f)= \int_{I^d} f(u) \log(f(u)) \, \mu(du).
\]
From Proposition \ref{prop:0} we can deduce that if $c$ is feasible then $\ci^\mu(c)=\ci_d(c)$.
Let us also define, for $1 \leq i \leq d$, $r \in I$:

\[
  \ca^\mu_i(c)(r) = \int_{I^d} c(u) \ind_{\{u_i \leq r \}} \, \mu(du), \quad
  \text{and} \quad
  \ca^\mu_{d+1}(c)(r) = \int_{I^d} c(u) \ind_{\{\max(u) \leq r \}} \, \mu(du).
\]
The corresponding optimization problem $(P^{\delta}_\mu)$ is given by :
\begin{equation}
\tag{$P^{\delta}_\mu$}
\text{minimize } \ci^\mu(c) \text{ subject to } 
\begin{cases}
 &  \ca^\mu(c)=b^\delta,\\
&c\geq 0 \text{ $\mu$-a.e. and } c\in L^1(I^d, \mu),
\end{cases}
\end{equation}
with $\ca^\mu=\left(\ca^\mu_i, 1\leq i \leq d+1\right)$.
 For $f \in L^1(I^d, \mu)$, we define:
\begin{equation*}
f^\mu =  
\begin{cases}
 &  f \text{ on } Z_\delta^c, \\
&  0 \text{ on } Z_\delta.
\end{cases}
\end{equation*}
Using Proposition \ref{prop:0}, we easily get the following Corollary. 
\begin{cor}
 If $c$ is a solution of $(P_\mu^\delta)$, then $c^\mu$ is a solution of $(P^\delta)$.
 If $c$ is a solution of $(P^\delta)$, then it is also a solution of $(P^\delta_\mu)$.
\end{cor}

\section{Proof of Theorem \ref{theo:spec}} \label{sec:proof-theo-spec}
\subsection{Form of the optimal solution}
Let $(\ca^\mu)^*: L^\infty (I)^{d+1} \rightarrow L^\infty (I^d,
\mu)$ be the adjoint of $\ca^\mu$.  
 We will use Theorem 2.9. from \cite{borwein1994entropy} on abstract
 entropy minimization, which we recall here, adapted to the context of
 $(P^{\delta}_\mu)$. 

\begin{theo}[Borwein, Lewis and Nussbaum] \label{bor_lew_nuss}
 Suppose there exists  $ c>0$ $\mu$-a.e. which is feasible for
 $(P^{\delta}_\mu)$. Then there exists a unique optimal solution,
 $ c ^*$, to $(P^{\delta}_\mu)$. Furthermore, we have
 $ c^*>0$ $\mu$-a.e. and there exists a sequence $(\lambda^n,
 n\in \N^*)$ of elements of $L^\infty (I)^{d+1}$ such that:
\begin{equation}
   \label{eq:cvA}
\int_{I^d}  c^*(x) \left| (\ca^\mu)^*(\lambda^n)(x)
  -\log( c^*(x)) \right| \; \mu(dx) 
\; \xrightarrow[n\rightarrow \infty ]{} \;0.
\end{equation}
\end{theo}

We first compute
$(\ca^\mu)^*$. For $\lambda=(\lambda_i, 1\leq i \leq d+1) \in L^{\infty} (I)^{d+1} $ and $f\in
 L^\infty (I^d, \mu)$, we have:
\begin{align*}
\langle (\ca^\mu)^*(\lambda),f \rangle
&= \langle \lambda,\ca^\mu(f)  \rangle \\
&= \sum_{i=1}^d \int_I dr\, \lambda_i (r)\int_{I^d} f(x)\ind_{\{x_i\leq r\}} d \mu(x) 
+ \int_I dr\, \lambda_{d+1} (r)\int_{I^d} f(x)\ind_{\{\max(x) \leq r\}} d\mu(x) \\
&=\int_{I^d} d\mu(x) \, f(x)\left( \sum_{i=1}^d \Lambda_i(x_i) + \Lambda_{d+1}(\max(x))
\right),
\end{align*}
where we used  the definition of the adjoint
operator for the first equality, Fubini's theorem for the second, and the
following notation for
the third equality:
\[
\Lambda_i(x_i)= \int_I  \lambda_i(r)\ind_{\{r\geq x_i\}}\,  dr,
\quad\text{and}\quad
\Lambda_{d+1}(t)= \int_I  \lambda_{d+1}(r)\ind_{\{r\geq t\}}\,  dr. 
\]
Thus, we can set for $\lambda \in L^\infty(I)^{d+1}$ and $x \in I ^d$:
\begin{equation}
   \label{eq:L*}
(\ca^\mu)^*(\lambda)(x)=\sum_{i=1}^d \Lambda_i(x_i) + \Lambda_{d+1}(\max(x)).
 \end{equation} 
Now we are ready to prove that the optimal solution $c^*$ of $(P^\delta_\mu)$ is the product of 
measurable univariate functions.

\begin{lem}
 \label{lem:c*=ab}
 Let $\delta\in \cd_0$ such that $\Sigma_\delta=\{0,1\}$. 
 Suppose that there exists $c>0$ $\mu$-a.e.which is feasible for $(P^\delta_\mu)$. 
 Then there exist $a^*,b^*$ non-negative, measurable functions 
 defined on $I$ such that 
 \[
 c^*(u)=b^*(\max(u))\prod_{u_i \neq \max(u)} a^*(u_i) \quad \mu\text{-a.e.}
 \] with $a^*(s)=0$ if $ \delta'(s)=d$ and
 $b^*(s)=0$ if $\delta'(s)=0$. 
\end{lem}

\begin{proof}
 
 According to Theorem \ref{bor_lew_nuss},
there exists a sequence $(\lambda^n, n\in \N^*)$ of elements of
$L^\infty (I)^{d+1}$ such that the optimal solution, say $c^*$,  satisfies
\reff{eq:cvA}. This implies, thanks to \reff{eq:L*},
 that there exist $d+1$ sequences $(\Lambda_i^n, n\in \N^*, 1 \leq i \leq d+1)$
 of elements of $L^\infty (I)$ such that the
 following convergence holds in $L^1(I^d, c^* \mu)$:
\begin{equation}
   \label{eq:cvln}
 \sum_{i=1}^d\Lambda^{n}_i(u_i)+\Lambda^n_{d+1}(\max(u)) 
\; \xrightarrow[n\rightarrow \infty ] \;
\log(c^*(u)).
\end{equation}

Arguing as in Proposition \ref{prop:sym} and since $Z_\delta^c$ is 
symmetric, we deduce that $c^*$ is symmetric. Therefore we shall
only consider functions supported on the set 
$\triangle=\{u \in I^d; u_d=\max(u)\}$. The convergence 
\reff{eq:cvln} holds in $L^1(\triangle, c^*\mu)$. For simplicity, 
we introduce the functions $\Gamma_i^n \in L^\infty(I)$ defined by
$\Gamma_i^n =\Lambda_i^n$ for $1 \leq i \leq d-1$, and 
$\Gamma_d^n=\Lambda_d^n+\Lambda_{d+1}^n$. Then we have in 
$L^1(\triangle, c^* \mu)$:
 \begin{equation}
   \label{eq:cvlntr}
 \sum_{i=1}^d \Gamma_i^n(u_i) 
\; \xrightarrow[n\rightarrow \infty ] \;
\log(c^*(u)).
\end{equation}
  We first assume that there exist 
$\Gamma$ and $\Gamma_d$ measurable functions defined on $I$ such that
 $\mu$-a.e. on $\triangle$: 
 \begin{equation}
  \label{eq:Gam_lim}
  \sum_{i=1}^{d-1} \Gamma(u_i) +\Gamma_{d}(u_d) = \log(c^*(u)).
  \end{equation}
  The symmetric property of $c^*(u)$ seen in Proposition \ref{prop:sym} implies 
we can choose $\Gamma_i= \Gamma$ for $1\leq i \leq d-1$ up to adding a constant
to $\Gamma_d$. Set $a^*=\exp(\Gamma)$ and
$b^*=\exp(\Gamma_d)$ so that $\mu$-a.e. on $\triangle$:
\begin{equation}
   \label{eq:c=ab}
c^*(u)=b^*(u_d)\prod_{i=1}^{d-1} a^*(u_i).
\end{equation}
Recall $\mu(du)=\ind_{Z_\delta^c}  (u) \,
du$. From the definition \reff{eq:def-Z} of $Z_\delta$, we deduce that
without loss  of   generality,  we   can  assume   that   $a^*(u_i)=0$  if
$\delta'(u_i)=d$ and $b^*(u_d)=0$ if $\delta'(u_d)=0$.
Use the symmetry of $c^*$ to conclude.

To complete the proof, we now show that \reff{eq:Gam_lim} holds for $\Gamma$ and $\Gamma_d$ measurable functions.
 We introduce the notation
 $u_{(-i)}=(u_1,\hdots, u_{i-1},u_{i+1},\hdots, u_d) \in I^{d-1}$.
 Let us define the probability measure $P(dx)=c^*(x)\ind_{\triangle}(x)\mu(dx)/ \int_{\triangle} c^*(y)\mu(dy)$
 on $I^d$. We fix $j$, $1 \leq j \leq d-1$. In order to apply Proposition 2 of \cite{Ruschendorf1993369}, 
 we first check that $P$ is absolutely continuous with respect to $P^j_1 \otimes P^j_2$,
 where $P^j_1(du_{(-j)})= \int_{u_j \in I} P(du_{(-j)} du_j)$ and $P^j_2(du_j)= \int_{u_{(-j)} \in I^{d-1}} P(du_{(-j)} du_j)$
 are the marginals of $P$. Notice the following equivalence of measures:
 \begin{equation}
 \label{eq:mes_equiv}
  P(du) \sim \ind_{\triangle}(u) \prod_{i=1}^{d-1} \ind_{\{ \delta'(u_i) \neq d\}}\ind_{\{ \delta'(u_d) \neq 0\}}\, du.
 \end{equation}
Let $B \subset I^{d-1}$ be measurable. We have:
\[ 
  P_1(B) = 0 \Longleftrightarrow \int_{I^d} \ind_{\triangle}(u) 
  \prod_{i=1}^{d-1} \ind_{\{ \delta'(u_i) \neq d\}}\ind_{\{ \delta'(u_d) \neq 0\}} \ind_B(u_{(-j)}) \, du = 0.
\]
 By Fubini's theorem this last equiality is equivalent to:
\begin{equation}
 \label{eq:mes_fubini_P1}
 \int_{I^{d-1}} \prod_{i=1,i\neq j}^{d-1} \left(\ind_{\{ \delta'(u_i) \neq d\}} \ind_{\{u_i \leq u_d \}} \right)
 \ind_{\{ \delta'(u_d) \neq 0\}} 
 \ind_B(u_{(-j)}) \left( \int_I \ind_{\{ 0 \leq u_j \leq u_d\}}\ind_{\{ \delta'(u_j) \neq d\}}\, du_j \right)\, du_{(-j)} = 0.
\end{equation}
Since, for $\varepsilon > 0$, $\delta(\varepsilon) <  \epsilon < d\varepsilon $, we have 
$\int_I \ind_{\{ 0 \leq u_j \leq s\}}\ind_{\{ \delta'(u_j) \neq d\}} \, du_j > 0$ for all $s \in I$.
Therefore \eqref{eq:mes_fubini_P1} is equivalent to 
\[
 \int_{I^{d-1}} \prod_{i=1,i\neq j}^{d-1} \left(\ind_{\{ \delta'(u_i) \neq d\}} \ind_{\{u_i \leq u_d \}} \right)
 \ind_{\{ \delta'(u_d) \neq 0\}} 
 \ind_B(u_{(-j)}) \, du_{(-j)} = 0.
\]
This implies that there exists $h>0$ a.e. on $I^{d-1}$ such that 
\[
P^j_1(du_{(-j)})=h(u_{(-j)})\prod_{i=1,i\neq j}^{d-1} \left(\ind_{\{ \delta'(u_i) \neq d\}} \ind_{\{u_i \leq u_d \}} \right)
 \ind_{\{ \delta'(u_d) \neq 0\}}  du_{(-j)}.
\]
Similarly we have for $B' \subset I$ that $P^j_2(B')=0$ if and only if
\begin{equation}
 \label{eq:mes_fubini_P2}
 \int_I \ind_{\{ \delta'(u_j) \neq d\}} 
 \ind_{B'}(u_j) \left( \int_{I^{d-1}} \prod_{i=1,i\neq j}^{d-1} 
 \left(\ind_{\{ \delta'(u_i) \neq d\}} \ind_{\{u_i \leq u_d \}} \right)
 \ind_{\{ \delta'(u_d) \neq 0\}} \, \ind_{\{u_d \geq u_j\}} du_{(-j)}  \right)\, du_j = 0.
\end{equation}
Since, for $\epsilon > 0$, $\delta(1)-\delta(1-\epsilon) > 1- (1-\epsilon) = \epsilon >0$ ,
there exists $g>0$ a.e. on $I$ such that $P^j_2(du_j)=g(u_j)\ind_{\{\delta'(u_j) \neq d \}} du_j$.
Therefore by \eqref{eq:mes_equiv} we deduce that $P$ is absolutely continuous with respect to $P^j_1 \otimes P^j_2$.
Then according to Proposition 2 of \cite{Ruschendorf1993369}, \eqref{eq:cvlntr} implies that
there exist measurable functions $\Phi_j $ and
$\Gamma_j $ defined respectively on $I^{d-1}$ and $I$, such that $c^*\mu$-a.e. on $\triangle$:
\[
\log(c^*(u))=\Phi_j(u_{(-j)})+\Gamma_j(u_j). 
\]
As    ${\mu}$-a.e. $c^*>0$,  this equality   holds $\mu$-a.e. on $\triangle$.
Since we have such a representation for every $1 \leq j \leq d-1$, we can easily
verify that there exists a measurable function $\Gamma_d$ defined on $I$
such that $\log(c^*(u))=\sum_{i=1}^d \Gamma_i(u_i)$ $\mu$-a.e. on $\triangle$.

\end{proof}

\subsection{Calculation of the optimal solution}
Now we prove that the optimal solution to $(P^{\delta})$, 
if it exists, is indeed $c_{\delta}$. 

\begin{prop}
   \label{prop:c-opt}
Let $\delta\in \cd_0$ such that $\Sigma_\delta=\{0,1\}$. 
If there exists an optimal solution to $(P^{\delta})$, then it
is  $c_\delta$ given by \reff{eq:cd_def}. 
\end{prop}

\begin{proof}
In Lemma \ref{lem:c*=ab} we have already shown that if an optimal solution exists
for $(P^{\delta})$, then it is of the form $c^*(u)=b^*(\max(u))\prod_{u_i \neq \max(u)} a^*(u_i)$. 
Here we will prove that the constraints of $(P^{\delta})$ uniquely 
determine the functions $a^*$ and $b^*$ up to a multiplicative constant, giving $c^*=c_\delta$.
We set for $r\in I$:
\[
A^*(r)=\int_0^r a^*(s)\,ds 
\]
which take values in $[0,+\infty ]$. 
From $\ca_{d+1}(c^*)=b^\delta_{d+1}$, we have for $r \in I$:
\begin{align}\label{eq:A^(d-1)del}
 \nonumber \delta(r) &= \int_{I^d} c^*(u)\ind_{\{\max(u)\leq r\}} \, du \\
 \nonumber &= \int_{I^d} b^*(\max(u)) \prod_{u_i \neq \max(u)}a^*(u_i)  \ind_{ \{\max(u) \leq r \} }  \, du \\
  &= d \int_{I} (A^*(s))^{d-1} b^*(s) \ind_{ \{s \leq r \} }\, ds.
\end{align}
Taking the derivative with respect to $r$ gives a.e. on $I$: 
 \begin{equation}\label{eq:C1}
  \delta'(r) = d(A^*(r))^{d-1} b^*(r).
 \end{equation}
This implies that $A^*(r)$ is
finite for  all $r\in [0,1)$  and thus $A^*(0)=0$.  Similarly,  using that
$\ca_1(c^*)=b^\delta_1$, we get that for $r\in I$:
\begin{align*}
  1-r & = \int_{I^d} c^*(u)\ind_{\{u_1\geq r\}} \, du \\
      & = \int_{I^d} b^*(\max(u)) \prod_{u_i \neq \max(u)} a^*(u_i)   \ind_{ \{u_1 \geq r \} }\, du\\
      & = \int_{I^d} \prod_{i=2}^{d} \left(a^*(u_i) \ind_{\{u_i \leq u_1\}} \right) 
              b^*(u_1) \ind_{ \{u_1 \geq r \} }\, du\\
      & \hspace{1.5cm} + (d-1) \int_{I^d} a^*(u_1) \prod_{i=3}^{d} \left(a^*(u_i) \ind_{\{u_i \leq u_2\}} \right) 
              b^*(u_2) \ind_{ \{u_2 \geq u_1 \geq r \} }\, du\\
      & = \int_I (A^*(s))^{d-1} b^*(s) \ind_{ \{s \geq r \} }\, ds \\
      & \hspace{1.5cm} + (d-1) \int_I (A^*(s))^{d-2} b^*(s)(A^*(s)-A^*(r)) \ind_{ \{s \leq r \} }\, ds \\
      & = d \int_I (A^*(s))^{d-1} b^*(s) \ind_{ \{s \geq r \} }\, ds - (d-1) A^*(r)\int_I (A^*(s))^{d-2}b^*(s) \ind_{ \{s \geq r \} }\, ds.
\end{align*}
Using this and \reff{eq:A^(d-1)del} we deduce that for $r\in I$:
\begin{equation}
 \label{eq:C2}
 h(r) = (d-1) A^*(r) \int_I (A^*(s))^{d-2}b^*(s) \ind_{ \{s \geq r \} }\, ds.
\end{equation}
Since $r>\delta(r)$ on $(0,1)$, we have that $A^*$ and $\int_I (A^*(s))^{d-2} b^*(s) 
\ind_{ \{s \geq r \} }\, ds$ are positive
on $(0,1)$. Dividing \reff{eq:C1} by \reff{eq:C2} gives a.e. for $r \in I$: 
\begin{equation*} 
 \frac{d-1}{d} \frac{\delta'(r)}{h(r)}= 
 \frac{(A^*(r))^{d-2} b(r)}{\int_I (A^*(r))^{d-2}b^*(s) \ind_{\{r \leq s \leq 1\}} \, ds} \cdot
\end{equation*}
We integrate both sides to get for $r \in I$:
\begin{equation*}
 \frac{d-1}{d}\left( \log\left(\frac{h(r)}{h(1/2)}\right) 
 -\int_{1/2}^r \frac{1}{h(s)} \, ds \right)= 
 \log \left(\frac{\int_I (A^*(s))^{d-2} b^*(s) \ind_{\{r \leq s \leq 1\}} \, ds}
 {\int_I (A^*(s))^{d-2} b^*(s) \ind_{\{1/2 \leq s \leq 1\}} \, ds} \right).
\end{equation*}
Taking the exponential yields:
\begin{equation} \label{eq:Aab}
 \alpha h^{(d-1)/d}(r) \expp{-F(r) } = 
 \int_I (A^*(s))^{d-2} b^*(s) \ind_{\{r \leq s \leq 1\}} \, ds,
\end{equation} 
 for some positive constant $\alpha$. From \reff{eq:C2} and \reff{eq:Aab}, we derive:
 \begin{equation}\label{eq:A(r)}
 A^*(r) = \frac{1}{\alpha(d-1)} h^{1/d} (r)\expp{F(r) }.
 \end{equation}
 This proves that the function $A^*$ is uniquely determined up to a 
 multiplicative constant and so is $a^*$. 
 With the help of \reff{eq:C1} and \reff{eq:A(r)}, we can express $b^*$ as, for $r \in I$:
 \begin{equation}
  b^*(r)=\frac{\delta'(r)(\alpha(d-1))^{d-1}}{d} 
 \expp{-(d-1)F(r)}.
 \end{equation}
  The function $b^*$ is also uniquely determined up to a multiplicative
constant. Therefore \reff{eq:C1} implies that there is a unique $c^*$ 
of the form \reff{eq:c=ab} which solves $\ca(c)=b^\delta$.
(Notice however that the functions $a^*$ and $b^*$ are defined up to a multiplicative
constant.) Then according to Proposition \ref{lem:cd=density} we get that $c_\delta$ defined 
by \reff{eq:c=ab} with $a$ and $b$ defined by
\reff{eq:a-b} solves  $\ca(c)=b^\delta$, implying
that $c^*$ is equal to  $c_\delta$. 
\end{proof}

\subsection{Proof of Theorem \ref{theo:spec}}
Let $\delta \in \cd_0  $ such that $\Sigma_{\delta}=\{0,1\}$.  Thanks to
Proposition  \ref{prop:c-opt},  we
deduce that if there exists an optimal solution to  $(P^\delta)$ then it is
$c_\delta$ given by \reff{eq:c=ab}.  By construction, we have $\mu$-a.e.
$c_\delta>0$.   According  to  Corollary \ref{lem:finiteKL},  $c_\delta$  is
feasible  for  $(P^\delta)$ if  and  only  if  $\cj(\delta) <  +\infty$.
Therefore if  $\cj(\delta) < +\infty$,  then $c_\delta$ is  the optimal
solution. If  $\cj(\delta)=+\infty  $  then there  is  no  optimal
solution.

\section{Proof of Theorem \ref{theo:gen}}
\label{sec:proof-theo-gen}
We first state an elementary Lemma, whose proof if left to the reader. 
For $f$ a function defined on $I^d$ and   $0\leq s<t\leq 1$, we define
$f^{s,t}$ by,  for $u \in I^d$:
\[
f^{s,t}(u)=(t-s) f(s \mathbf{1}+ u(t-s)).
\]
\begin{lem}
   \label{lem:scaling}
If $c$ is the density of a copula $C$ such that
$\delta_C(s)=s$ and $\delta_C(t)=t$ for some fixed 
$0 \leq s < t \leq 1$, then $c^{s,t}$ is also the density
of a copula, and its diagonal section, $\delta^{s,t}$, is given by, for
$r\in I$:
\[
\delta^{s,t}(r)= \frac{\delta_C(s+ r(t-s))- s}{t-s} \cdot
\]
\end{lem}

According  to Remark \ref{rem:abs-cont},  it is  enough to  consider the
case  $\delta\in \cd_0$,  that is  $\Sigma_\delta$ with zero Lebesgue measure.
We shall assume that $\Sigma_{\delta} \neq \{ 0,1 \}$.
Since $\delta$ is continuous,  we get that $I \setminus \Sigma_{\delta}$
can be  written as  the union of  non-empty open  intervals $((\alpha_j,
\beta_j),   j\in  J)$,   with   $\alpha_j<\beta_j$  and   $J$
non-empty and  at   most
countable. Set $\Delta_j=\beta_j-\alpha_j$. Since $\Sigma_\delta$ is of
zero Lebesgue measure, we have $\sum_{j\in J} \Delta_j=1$. 
We define also $S=\bigcup _{j\in J} [\alpha_j, \beta_j]^d$

For $s\in \Sigma_\delta$, notice that any feasible function $c$ of
$(P^\delta)$ satisfies for all $1\leq i \leq d$:
\begin{equation*}
\int_{I^d} c(u) \ind_{\{u_i < s\}} \ind_{D_i^c}(u)\, du  = \int_{I^d} c(u) \ind_{\{u_i < s\}} \, du
- \int_{I^d} c(u) \ind_{\{\max(u) < s\}} \, du = s-\delta(s)=0,
\end{equation*}
where $D_i=\{u \in I^d \text{ such that } \forall j \neq i :  u_{j} < s\}$.
This implies that $c=0$ a.e. on $I^d \setminus S$. 
We set $c^j=c^{\alpha_j,
  \beta_j} $ for $j\in J$. We 
deduce that if $c$ is feasible for $(P^\delta)$, then we have that a.e.:
\begin{equation}
   \label{eq:cj}
c(u)=\sum_{j\in J} \inv{\Delta_j} c^j\left(\frac{u-\alpha_j\mathbf{1}}{\Delta_j} 
\right) \,\ind_{(\alpha_j,\beta_j)^d}(u),
\end{equation}
and: 
\begin{equation}
   \label{eq:cidc}
\ci_d(c)=\sum_{j\in J} \Delta_j \left(\ci_d(c^j) - 
\log(\Delta_j)\right). 
\end{equation}
Thanks  to Lemma \ref{lem:scaling},  the condition  $\ca(c)=b^\delta$ is
equivalent to $\ca(c^j)=b^{\delta^j}$ for  all $j\in J$.  We deduce that
the  optimal  solution  of  $(P^\delta)$,  if it  exists,  is  given  by
\reff{eq:cj},  where the functions  $c^j$ are  the optimal  solutions of
$(P^{\delta^j})$   for   $j\in   J$.    Notice  that   by   construction
$\Sigma_{\delta^j}=\{0,1\}$.   Thanks  to  Theorem \ref{theo:spec},  the
optimal  solution to  $(P^{\delta^j})$ exists  if  and only  if we  have
$\cj(\delta^j)<+\infty  $;  and  if  it   exists  it  is  given  by  $c_
{\delta^j}$.  Therefore,   if  there  exists  an   optimal  solution  to
$(P^\delta)$, then  it is $c_\delta$ given  by \reff{eq:density-gen}. To
conclude,   we   have   to   compute  $\ci_d(c_\delta)$.   Recall   that
$x\log(x)\geq -1/\expp{}$ for $x>0$. We have:
\begin{align*}
\ci_d(c_\delta)
&=\lim_{\varepsilon\downarrow 0} \sum_{j\in J} \Delta_j \left(\ci_d(c^j) - 
\log(\Delta_j)\right)\ind_{\{\Delta_j>\varepsilon\}}\\
&=\lim_{\varepsilon\downarrow 0} \sum_{j\in J} \Delta_j \left((d-1)\cj(\delta^j) - 
\log(\Delta_j)\right)\ind_{\{\Delta_j>\varepsilon\}} +
\sum_{j\in J} \Delta_j  \cg(\delta^j) \\
&= \sum_{j\in J} \Delta_j \left((d-1)\cj(\delta^j) - 
\log(\Delta_j)\right)+ 
\sum_{j\in J} \Delta_j  \cg(\delta^j),
\end{align*}
where we used the monotone  convergence theorem for the first equality,
\reff{eq:ci=cj+g} for the second and the fact that $\cg(\delta)$
is uniformly bounded over  $\cd_0$ and the monotone convergence theorem
for the last.  Elementary computations yields:
\[
(d-1)\cj(\delta)=\sum_{j\in J} \Delta_j\left((d-1)\cj(\delta^j) -
  \log(\Delta_j)\right) 
\quad\text{and}\quad
\cg(\delta)= \sum_{j\in J} \Delta_j  \cg(\delta^j). 
\]
So, we get:
\[
\ci_d(c_\delta)=(d-1)\cj(\delta)+ \cg(\delta).
\]
Since $\cg(\delta)$ is
uniformly bounded over $\cd_0$, we get that $\ci_d(c_\delta)$ is finite if
and only if $\cj(\delta)$ is finite. 
To end the proof, recall  the definition of $\ci(C_\delta)$ to conclude that
$\ci(C_\delta)=(d-1)\cj(\delta)+ \cg(\delta)$.

   \section{Examples for $d=2$} \label{sec:examples}
   
   In this section we compute the density of the maximum entropy copula for various diagonal sections of 
   popular bivariate copula families. In this Section, $u$ and $v$ will denote elements of
   $I$. The density for $d=2$ is of the form $c_\delta(u,v)=a(\min(u,v))b(\max(u,v))$.
   We illustrate these densities by displaying their isodensity lines or contour plots, and their diagonal
   cross-section $\varphi$ defined as $\varphi(t)=c(t,t)$, $t \in I$. 
   
   \subsection{Maximum entropy copula for a piecewise linear diagonal section} \label{sec:ex-pie-lin}
   
   Let $\alpha\in (0, 1/2]$. Let us calculate the density of the maximum entropy copula in the case of 
   the following diagonal section:
   
\[
\delta(r)= (r-\alpha)\ind_{(\alpha, 1-\alpha)}(r)+ (2r-1)
\ind_{[1-\alpha, 1]}(r).
\]
   This example was considered for example in \cite{Nelsen2004348}. The limiting cases $\alpha=0$ and $\alpha=1/2$
   correspond to the Fr\'{e}chet-Hoeffding upper and lower bound copulas, respectively. 
   However for $\alpha=0$, $\Sigma_\delta=I$, therefore every copula $C$ with this diagonal section gives $\ci(C)=+ \infty$. (In fact
   the only copula that has this diagonal section is the Fr\'{e}chet-Hoeffding upper bound $M$ defined by $M(u,v)=
   \min(u,v)$, $u,v \in I$.)
   When $\alpha\in (0, 1/2]$, $\cj(\delta) < + \infty$ is satisfied, therefore we can apply Theorem \ref{theo:spec}
   to compute the density of the maximum entropy copula. The graph of $\delta$ can be seen in Figure~\ref{delta_a} for $\alpha=0.2$. We compute the functions 
   $F$,$a$ and $b$:
   
  \begin{figure} 
   
  \centering
  
   \begin{tikzpicture}[scale=2]
  
  \draw[->] (-1.5,-1.2) -- (1.5,-1.2)  node[right] {$t$}  ;
  \draw[->] (-1.2,-1.5) -- (-1.2,1.5)  node[above] {$\delta(t)$} ;
  \draw[very thick] (-1.2,-1.2) -- (-0.72,-1.2) node[below] {$\alpha$} ;
  \draw[thick] (-0.72,-1.2) -- (0.72,0.24)  ;
  \draw[thick] (0.72,0.24) -- (1.2,1.2) ;
  
  \draw (0.72,-1.15)--(0.72,-1.2) node[below] {$1-\alpha$};
  \draw[thin] (1.2,1.2) -- (1.2,-1.2) node[below] {$1$};
   \draw[thin] (1.2,1.2) -- (-1.2,1.2) node[left] {$1$};
  \draw[thin] (-1.2,-1.2) -- (1.2,1.2) ;
  
\end{tikzpicture}

\caption{Graph of $\delta$ with $\alpha=0.2$.}
\label{delta_a}
 \end{figure}
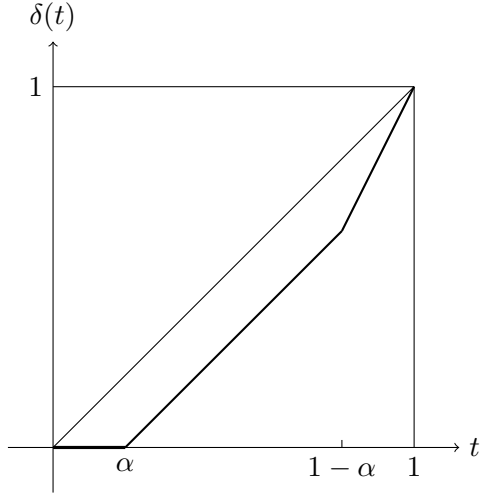 
   
   $$F(r) = \left\{ 
	\begin{array}{ll}
		\inv{2}\log(\frac{r}{\alpha}) -\frac{1}{4\alpha}+\inv{2} & \textrm{if } r \in [0,\alpha) \\
		\frac{r}{2\alpha}-\frac{1}{4\alpha} & \textrm{if } t \in [\alpha,1-\alpha)  \\
		\inv{2}\log \left( \frac{\alpha}{1-r} \right)+\frac{1}{4\alpha}-\inv{2}  & \textrm{if } t \in [1-\alpha,1]
	\end{array}
    \right.
    $$
   $$a(r)=\frac{1}{\sqrt{\alpha}}e^{-\frac{1}{4\alpha}+\frac{1}{2}}\ind_{[0,a]}(r) +  
   \frac{1}{2\sqrt{\alpha}}e^{\frac{r}{2\alpha}-\frac{1}{4\alpha}}\ind_{(a,1-a)}(r)    $$
   and: 
    $$b(r)=\frac{1}{2\sqrt{\alpha}}e^{-\frac{r}{2\alpha}+\frac{1}{4\alpha}}\ind_{(a,1-a)}(r) +
           \frac{1}{\sqrt{\alpha}}e^{-\frac{1}{4\alpha}+\frac{1}{2}}\ind_{[1-a,1]}(r) $$  
   The density $c_\delta(u,v)$ consists of six distinct regions on $\triangle=\{(u,v)\in I^2, u \leq v \}$ as shown
   in Figure~\ref{six_part}
   and takes the values:
  
  \begin{equation}
  \label{eq:ex_lin}
    c_\delta(u,v) = \left\{ 
	\begin{array}{ll}
		 0 & \textrm{in I,}   \\
		 \frac{1}{2\alpha}\expp{\frac{\alpha-v}{2\alpha}} & \textrm{in II,}   \\
		 \frac{1}{4\alpha}\expp{\frac{u-v}{2\alpha}} & \textrm{in III,}   \\
		 \frac{1}{\alpha}\expp{\frac{2\alpha-1}{2\alpha}}& \textrm{in IV,}   \\
		 \frac{1}{2\alpha} \expp{\frac{u+\alpha-1}{2\alpha}}  & \textrm{in V,}   \\
		 0 & \textrm{in VI.}   \\
		 
	\end{array}
    \right.
   \end{equation} 
   
   \begin{figure} 
   
   \centering
  
   \begin{subfigure}[b]{.45\linewidth}
   
   \centering
  
   \begin{tikzpicture}[scale=2]
  
   \draw[->] (-1.5,-1.2) -- (1.5,-1.2)    ;
   \draw[->] (-1.2,-1.5) -- (-1.2,1.5) ;

   \draw (-0.6,-1.15)--(-0.6,-1.2) node[below] {$\alpha$};
   \draw (0.6,-1.15)--(0.6,-1.2) node[below] {$1-\alpha$};

   \draw[thin] (1.2,1.2) -- (1.2,-1.2) node[below] {$1$};
   \draw[thin] (1.2,1.2) -- (-1.2,1.2) node[left] {$1$};
   \draw[thin] (-1.2,-1.2) -- (1.2,1.2) ;
   \draw[thin] (-1.2,-0.6) node[left] {$\alpha$} -- (-0.6,-0.6) ;
   \draw[thin] (-1.2,0.6) node[left] {$1-\alpha$} -- (0.6,0.6) ;
   \draw[thin] (-0.6,-0.6) -- (-0.6,1.2) ;
   \draw[thin] (0.6,0.6) -- (0.6,1.2) ;
   
   \draw (-1.0,-0.8) node { I.};
   \draw (-0.9,0) node { II.};
   \draw (-0.2,0.2) node { III.};
   \draw (-0.9,0.9) node { IV.};
   \draw (0.0,0.9) node { V.};
   \draw (0.8,1.0) node { VI.};
   
   \end{tikzpicture}
   \caption{Partition for $c_\delta$}
   \label{six_part} 
  
  \end{subfigure}
  \begin{subfigure}[b]{.45\linewidth}
       \centering
       \includegraphics[width=70mm]{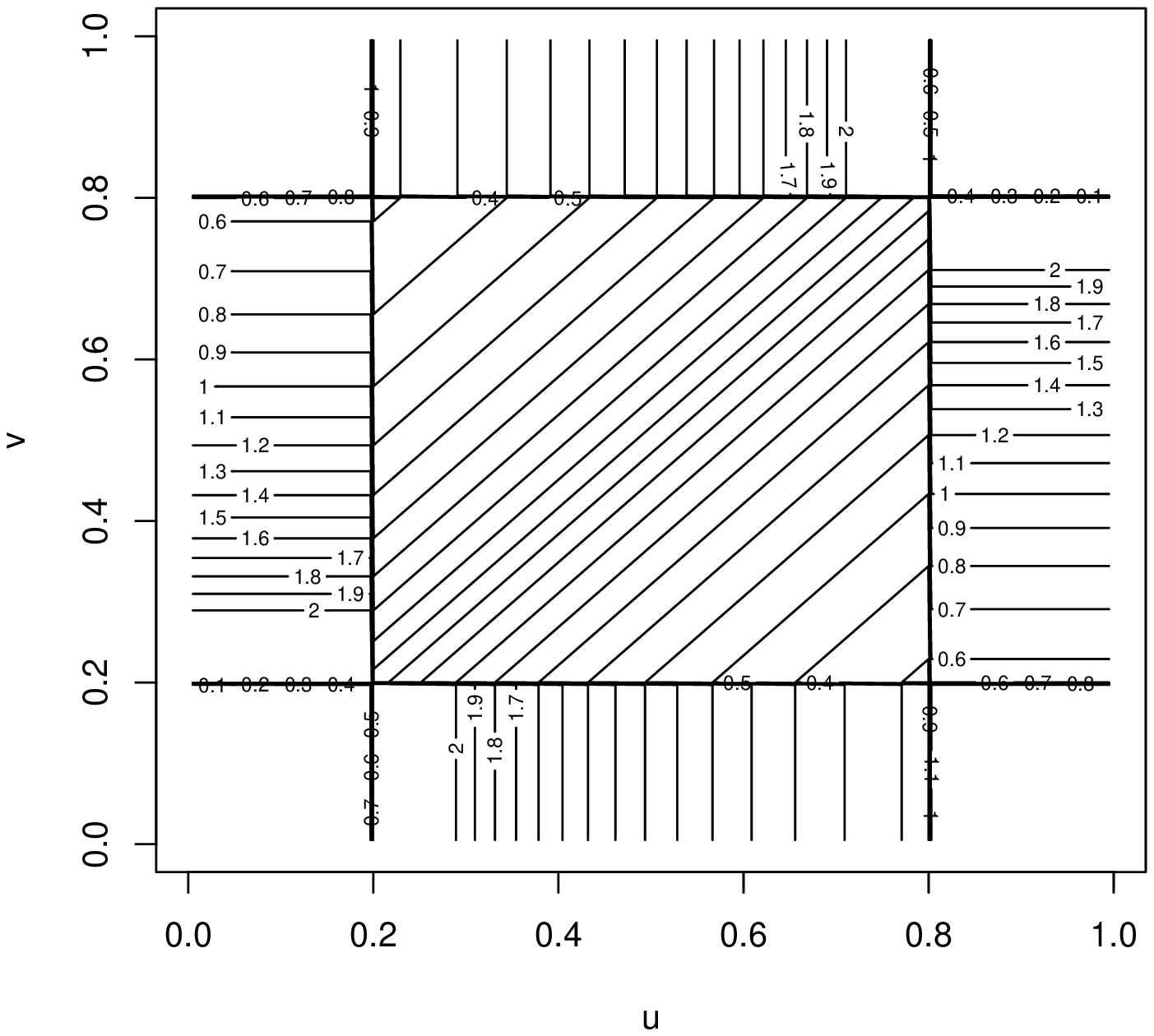}
       \caption{Isodensity lines of $c_\delta$}
       \label{SU_dens}     
  \end{subfigure}

     \caption{The partition and the isodensity lines of $c_\delta$.}
   \label{part_iso} 
  
   \end{figure} 

   Figure \ref{SU_dens} shows the isodensity lines of $c_\delta$. In the limiting case of $\alpha=\frac{1}{2}$, the diagonal section is given by $\delta(t)= \max(0,2t-1)$,which is the pointwise lower bound for all
   elements in $\mathcal{D}$. Accordingly, it is the diagonal section of the Fr\'{e}chet-Hoeffding lower bound copula $W$ given by
   $W(u,v)=\max(0,u+v-1)$ for $u,v \in I$. All copulas having this diagonal section are of the following form:
   
   $$D_{C_1,C_2}(u,v)=\left\{
	\begin{array}{ll}
		W(u,v)  &\textrm{if } (u,v) \in [0,1/2]^2 \cup [1/2,1]^2 , \\
		\frac{1}{2}C_1(2u,2v-1) & \textrm{if } (u,v) \in [0,1/2] \times [1/2,1],  \\
		\frac{1}{2}C_2(2u-1,2v) & \textrm{if } (u,v) \in   [1/2,1] \times [0,1/2] ,       
	\end{array}
\right.$$
where $C_1$ and $C_2$ are copula functions. Recall that the independent copula $\Pi$ with uniform
density $c_\Pi=1$ on $I^2$ minimizes $\ci(C)$ over $\cc$. According to \reff{eq:ex_lin}, the maximum
entropy copula with diagonal section $\delta$ is $D_{\Pi,\Pi}$. This corresponds to choosing
the maximum entropy copulas on $[0,1/2] \times [1/2,1] $ and $[1/2,1] \times [0,1/2]$.
    
\subsection{Maximum entropy copula for $\delta(t)=t^{\alpha}$ }
    Let $\alpha \in (1,2]$. We consider the family of diagonal sections given by $\delta(t)=t^{\alpha}$. This corresponds
    to the Gumbel family of copulas and also to the family of Cuadras-Aug\'{e} copulas. The Gumbel copula with parameter $\theta
    \in [1,\infty)$ is an Archimedean copula defined as, for $u,v \in I$: 
    $$C^G(u,v)= \phi^{-1}_{\theta}(\phi_{\theta}(u)+\phi_{\theta}(v))$$
    with generator function $\phi_{\theta}(t)= (-\log(t))^{\theta}$. Its diagonal section is given by $\delta^G(t)=t^{2^{\frac{1}{\theta}}}=t^{\alpha}$ with
    $\alpha=2^{\frac{1}{\theta}}$. The Cuadras-Aug\'{e} copula with parameter $\gamma \in (0,1)$ is defined as, for $u,v \in I$:
    $$C^{CA}(u,v)=\min(uv^{1-\gamma},u^{1-\gamma}v).$$ It is a subclass of the two parameter Marshall-Olkin family of copulas
    given by $C^M(u,v)=\min(u^{1-\gamma_1}v,uv^{1-\gamma_2})$. The diagonal section of $C^{CA}$ is given by $\delta(t)= t^{2-\gamma}=t^{\alpha}$ 
    with $\alpha=2-\gamma$. While the Gumbel copula is absolutely continuous, the Cuadras-Aug\'{e} copula is not, although it has full support.
    Since $\mathcal{J}(\delta)<+\infty$, we can apply Theorem \ref{theo:spec}.  
    To give the density of the maximum entropy copula, we have to calculate $F(v)-F(u)$. 
    Elementary computations yield:
    \[
       F(v)-F(u)
      =\inv{2}\int_u^v \frac{ds}{s-s^\alpha}
      =\inv{2}\log\left(\frac{v}{u}\right)
      - \inv{2\alpha-2} \log\left(\frac{1-v^{\alpha-1}}{1-u^{\alpha-1}}\right).
    \]
    The density $c_{\delta}$ is therefore given by, for $(u,v) \in \triangle$:
  \[
    c_\delta(u,v)=\frac{\alpha}{4}\,  \frac{2- \alpha
    u^{\alpha-1}}{(1-u^{\alpha-1})^{\alpha/(2\alpha-2)}} 
    \, v^{\alpha-2} (1-v^{\alpha-1})^{(2-\alpha)/(2\alpha-2)} .
  \]
   Figure \ref{pow_iso} represents the isodensity lines of the Gumbel and the maximum entropy copula $c_\delta$
   with common parameter $\alpha=2^{\frac{1}{3}}$, which corresponds to $\theta=3$ for the Gumbel copula. We have also added
   a graph of the diagonal cross-section of the two densities. 
   In the limiting case of $\alpha=2$, the above formula gives $c_\delta(u,v)=1$, which is the density of the 
   independent copula $\Pi$, which is also maximizes the entropy on the entire set of copulas.

   \begin{figure}[ht]
   
   \centering
   
   \begin{subfigure}[b]{.33\linewidth}
       \centering
       \includegraphics[width=50mm]{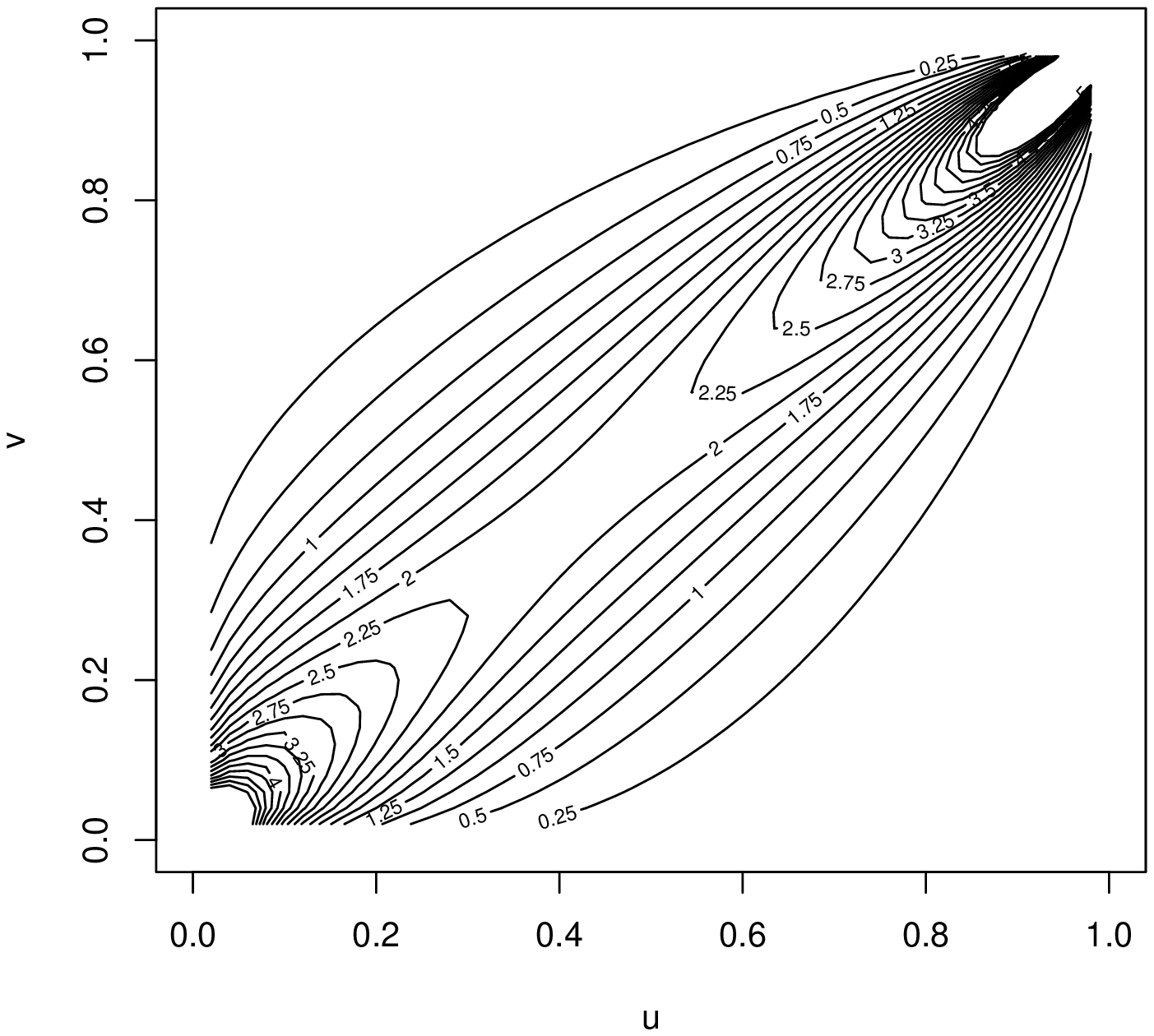}
       \caption{Gumbel}
       \label{power_gumbel}      
    \end{subfigure}%
     \begin{subfigure}[b]{.33\linewidth}
      \centering
      \includegraphics[width=50mm]{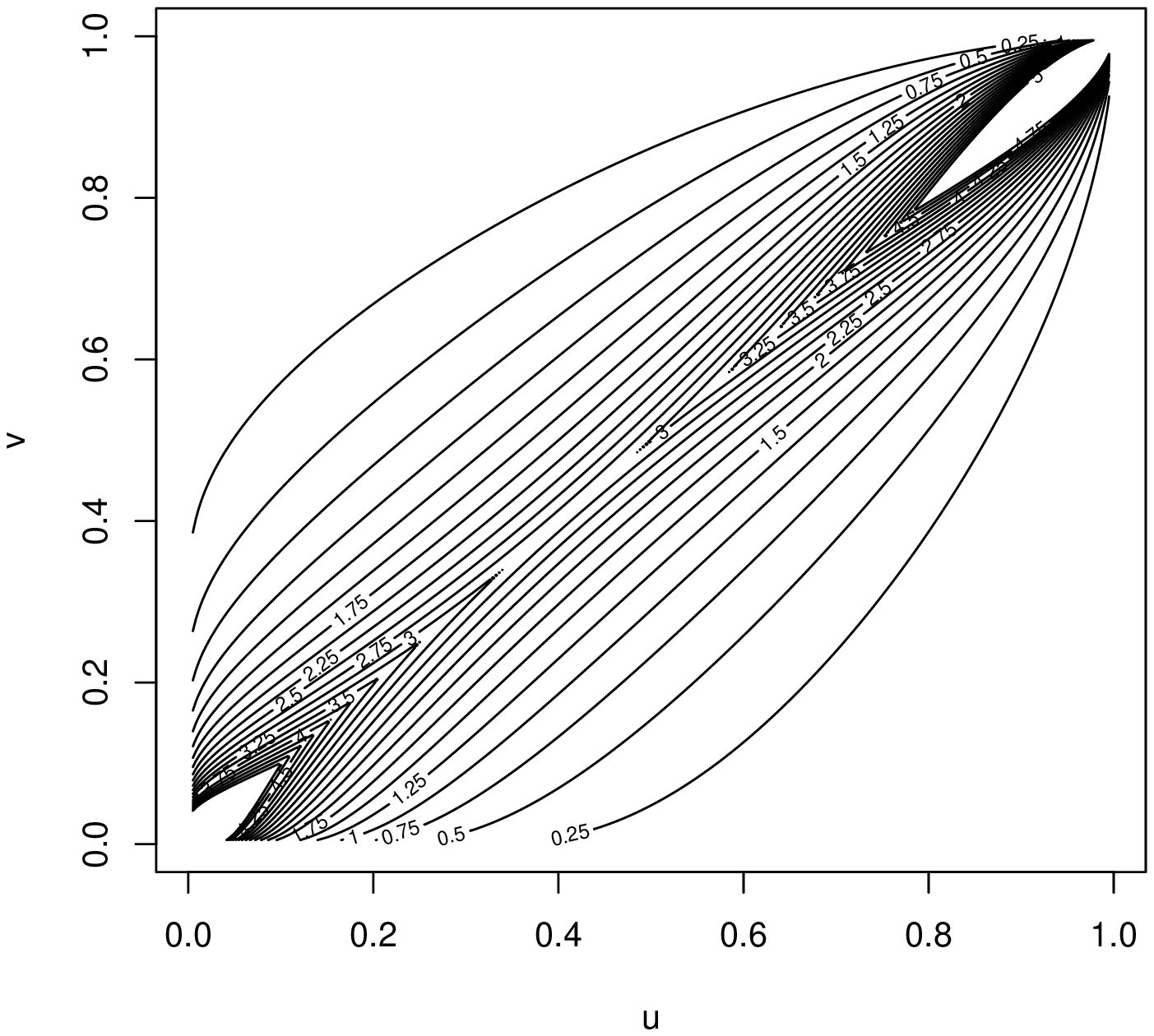}
      \caption{$c_\delta$}
      \label{power_max_ent}
   \end{subfigure}
    \begin{subfigure}[b]{.33\linewidth}
      \centering
      \includegraphics[width=50mm]{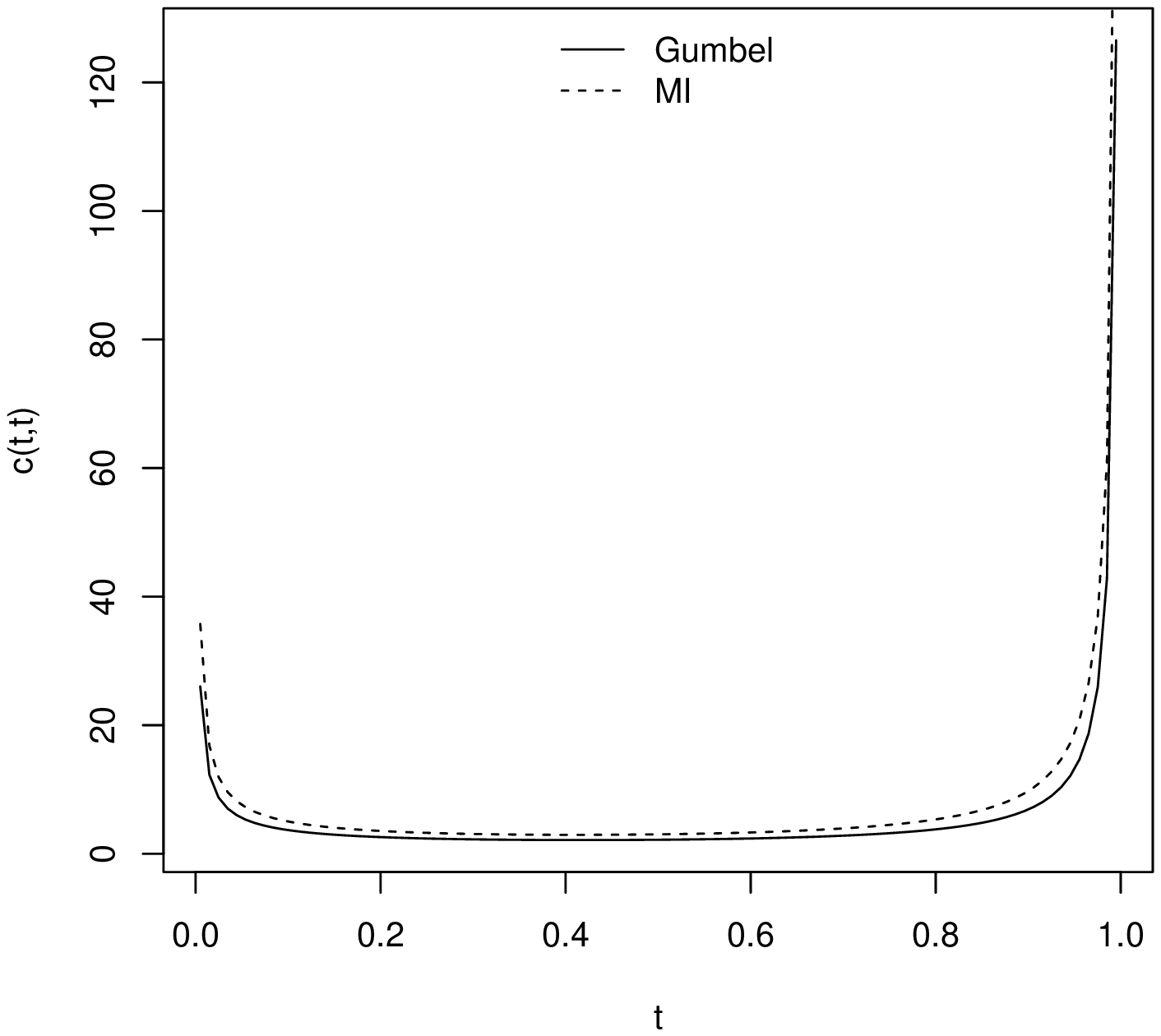}
      \caption{Diagonal cross-section}
      \label{power_dens_diag}
   \end{subfigure}
   \caption{Isodensity lines and the diagonal cross-section of copulas with diagonal section $\delta(t)=t^{\alpha}$, $\alpha=2^{\frac{1}{3}}$.}
   \label{pow_iso}  
   \end{figure}

 \subsection{Maximum entropy copula for the Farlie-Gumbel-Morgenstern diagonal section }
 
 Let $\theta \in [-1,1]$. The Farlie-Gumbel-Morgenstern family of copulas (FGM copulas for short) are defined as:
 $$C(u,v)=uv + \theta uv(1-u)(1-v) .$$
 These copulas are absolutely continuous with densities $c(u,v)=1+\theta(1-2u)(1-2v)$.
 Its diagonal section $\delta_{\theta}$ is given by:
 \begin{eqnarray*}
   \delta(t)&=&t^2+ \theta t^2(1-t)^2 \\
                     &=&\theta t^4-2\theta t^3+(1+\theta)t^2.
 \end{eqnarray*}
 
 Since $\delta_{\theta}(t) < t$ on $(0,1)$ and it verifies $\mathcal{J}(\delta) < + \infty $, we can apply
 Theorem \ref{theo:spec} to calculate the density of the maximum entropy copula. For $F(r)$, we have:
 
 \begin{eqnarray*}
  F(r)  =
 \left\{
	\begin{array}{ll}
		 \inv{2}\log\left(\frac{r}{1-r}\right)+ \frac{\theta}{\sqrt{4\theta-\theta^2}} \arctan \left( \frac{2\theta r - \theta}{\sqrt{4\theta-\theta^2}} \right)  &\textrm{if } \theta \in (0,1],  \\
                 \inv{2}\log\left(\frac{r}{1-r}\right) & \textrm{if } \theta=0,	\\
		\inv{2}\log\left(\frac{r}{1-r}\right)- \frac{\theta}{\sqrt{\theta^2-4\theta}} \arctanh \left( \frac{2\theta r - \theta}{\sqrt{\theta^2-4\theta}} \right) &\textrm{if } \theta \in [-1,0).         
	\end{array}
\right.
 \end{eqnarray*}
       
 The density $c_{\delta}$ is given by, for $\theta \in (0,1]$ and $(u,v) \in \triangle$:
 
   \begin{eqnarray*}
      c_{\delta}(u,v) &=&\frac{\left(1-2\theta u ^3+ 3 \theta u^2+(1+\theta) u\right)}{(1-u)\sqrt{\theta u^2 -\theta u+1}}
  \frac{\left(2\theta v^2+ 3 \theta v+(1+\theta)\right)}{\sqrt{\theta v^2-\theta v+1}} \\
     &&\hspace{0.5cm} \exp\left(-\frac{\theta}{\sqrt{4\theta-\theta^2}}\left(\arctan \left( \frac{2\theta v - \theta}{\sqrt{4\theta-\theta^2}} \right) 
   -\arctan \left( \frac{2\theta u - \theta}{\sqrt{4\theta-\theta^2}} \right)\right)\right)
   \end{eqnarray*} 

   Figure \ref{fgm_iso} illustrates the isodensities of the FGM copula and the maximum entropy copula 
   with the same diagonal section for $\theta=0.5$ as well as the diagonal cross-section of their densities.
 
    \begin{figure}[ht]
   
   \centering
   
   \begin{subfigure}[b]{.32\linewidth}
       \centering
       \includegraphics[width=50mm]{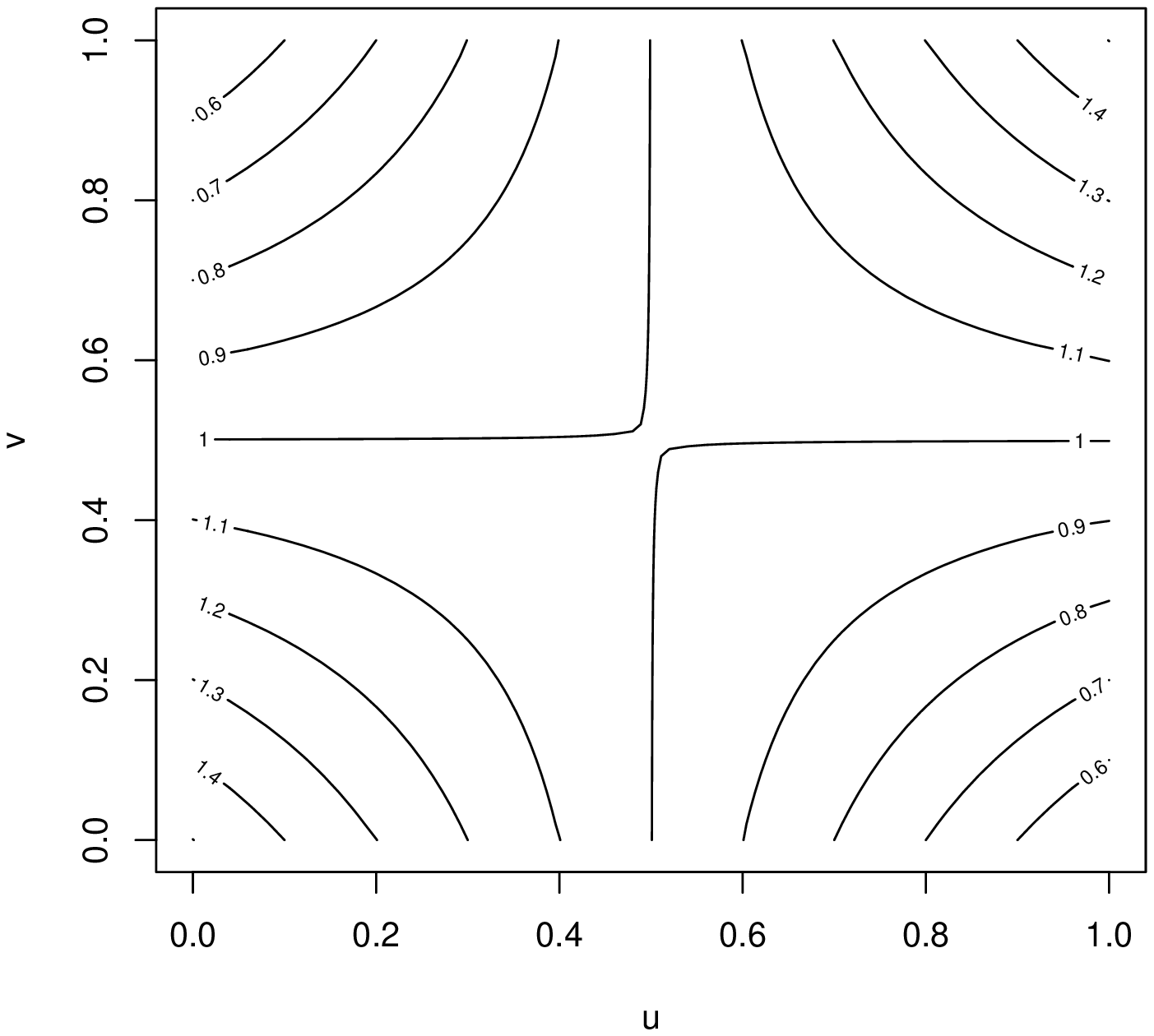}
       \caption{FGM}
       \label{FGM}      
    \end{subfigure}
     \begin{subfigure}[b]{.32\linewidth}
      \centering
      \includegraphics[width=50mm]{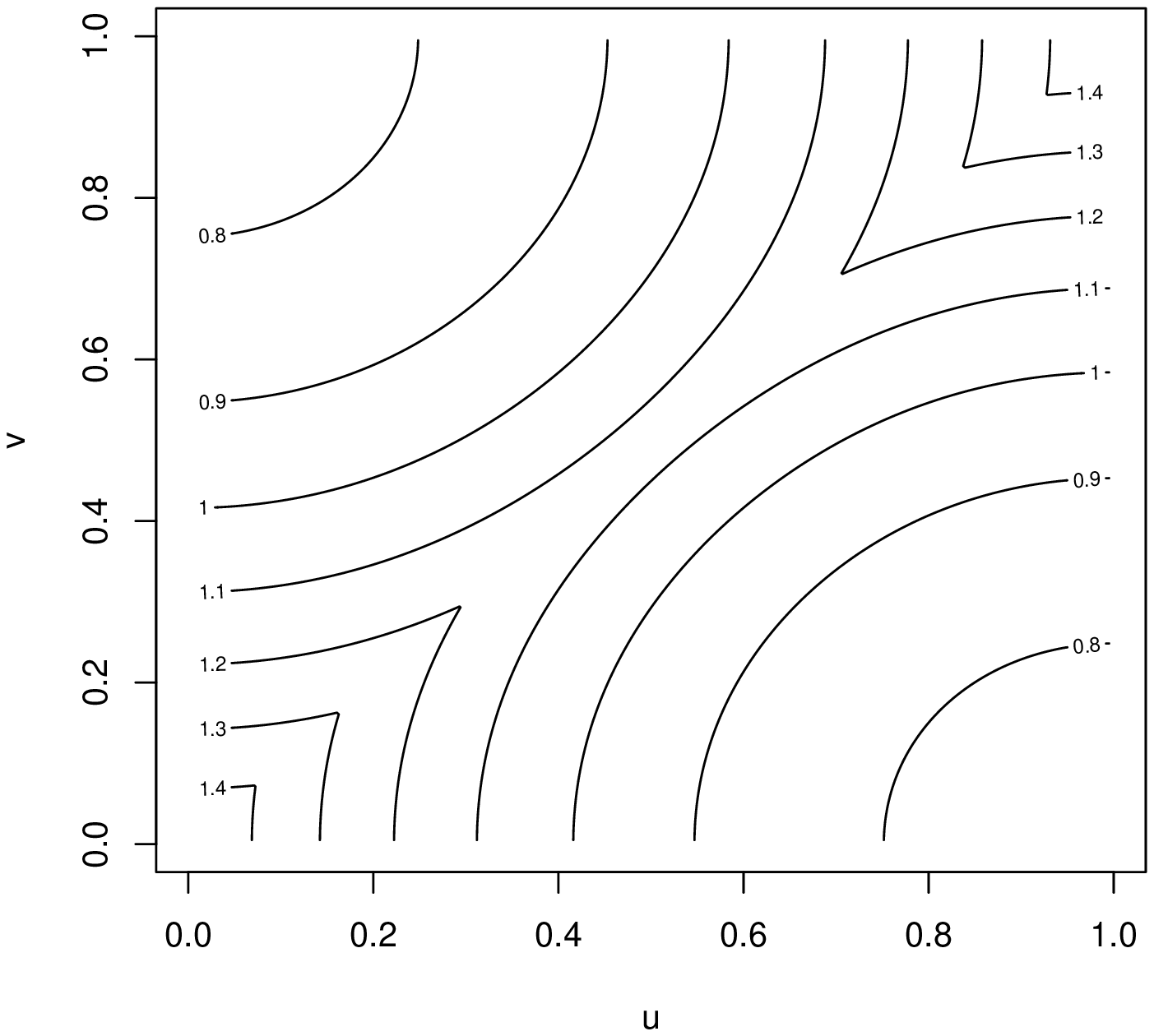}
      \caption{$c_\delta$}
      \label{FGM_max_ent}
   \end{subfigure}
   \begin{subfigure}[b]{.32\linewidth}
      \centering
      \includegraphics[width=50mm]{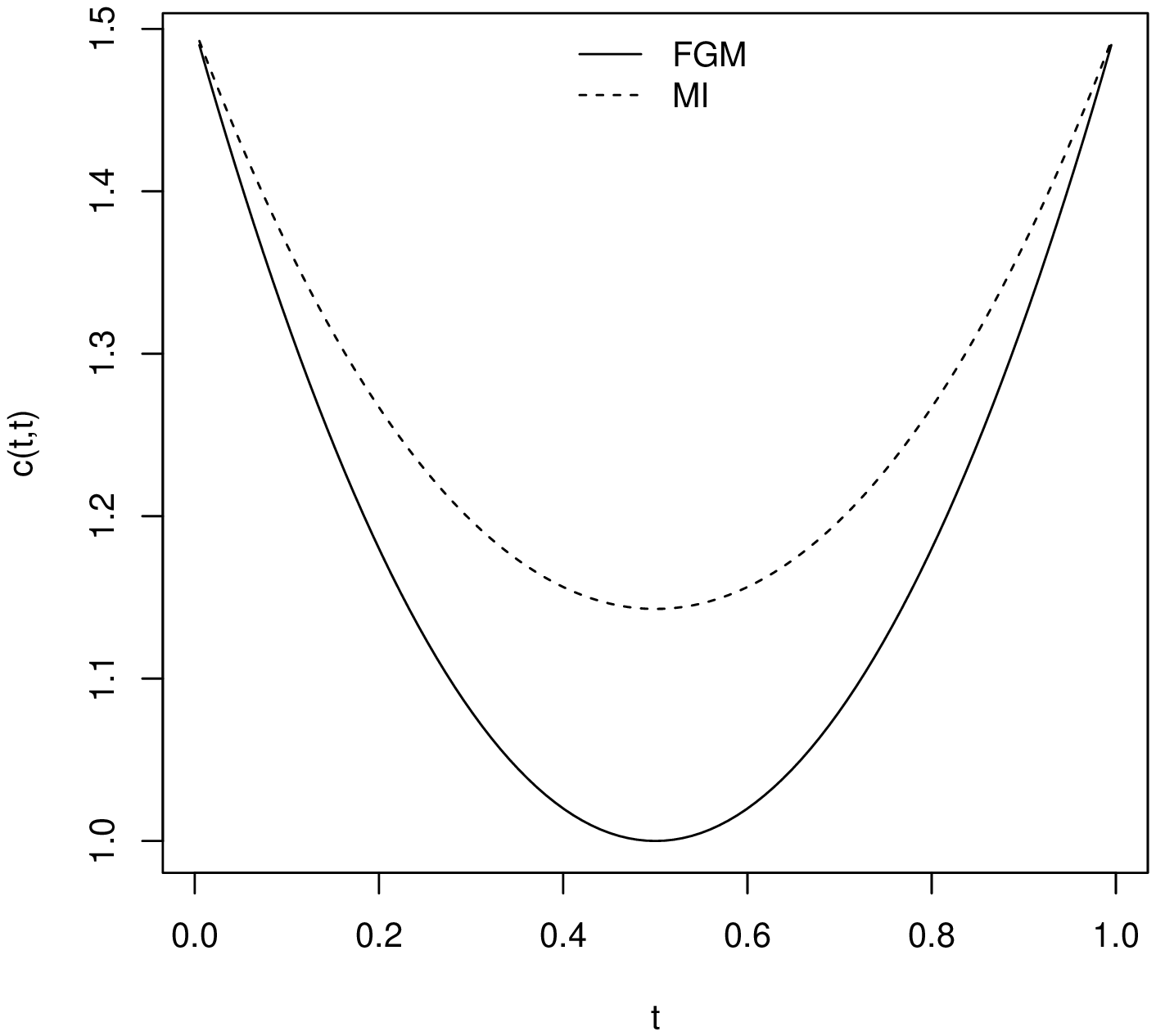}
      \caption{Diagonal cross-section}
      \label{FGM_dens_diag}
   \end{subfigure}
   \caption{Isodensity lines and the diagonal cross-section of copulas with diagonal section $\delta(t)=\theta t^4 -2\theta t^3+(1+\theta)t^2$, $\theta=0.5$.}
   \label{fgm_iso}  
   \end{figure} 
 
 The case of $\theta=0$ corresponds once again to the diagonal section $\delta(t)=t^2$, and the formula gives the density of the
 independent copula $\Pi$, accordingly.
 
 \subsection{Maximum entropy copula for the Gaussian diagonal section}
 
 The Gaussian (normal) copula takes the form:
 $$C_{\rho}(u,v)= \Phi_{\rho}\left( \Phi^{-1}(u),\Phi^{-1}(v) \right),$$
 with $\Phi_{\rho}$ the joint cumulative distribution function of a two-dimensional normal random variable
 with standard normal marginals and correlation parameter $\rho \in [-1,1]$, and $\Phi^{-1}$ the quantile function of the standard normal
 distribution. The density $c_{\rho}$ of $C_{\rho}$ can be written as:
 $$c_{\rho}(u,v)= \frac{\phi_{\rho}\left( \Phi^{-1}(u),\Phi^{-1}(v) \right)}{\phi(\Phi^{-1}(u))\phi(\Phi^{-1}(v))} ,$$
 where $\phi$ and $\phi_{\rho}$ stand for respectively the densities of a standard normal distribution and a two-dimensional normal distribution
 with correlation parameter $\rho$, respectively. The diagonal section and its derivative are given by:
 \begin{equation}
 \label{eq:delta_normal}
  \delta_{\rho}(t) = \Phi_{\rho}\left( \Phi^{-1}(t),\Phi^{-1}(t) \right), \ \delta_{\rho}'(t) = 2\Phi \left( \sqrt{\frac{1-\rho}{1+\rho} }\Phi^{-1}(t) \right).
 \end{equation}
 Since $\delta_\rho$ verifies $\delta_{\rho}(t) < t$ on $(0,1)$ and $\mathcal{J}(\delta_\rho) < + \infty $, we can apply
 Theorem \ref{theo:spec} to calculate the density of the maximum entropy copula. We have calculated numerically
 the density of the maximum entropy copula with diagonal section $\delta_{\rho}$ for 
 $\rho=0.95,0.5,-0.5$ and $-0.95$. The comparison between these densities and the densities of the corresponding normal 
 copula can be seen in Figures \ref{normal_iso},\ref{normal_iso_neg} and \ref{normal_neg_sample}. We observe a very different behaviour of $c_\rho$ and 
 $c_{\delta_\rho}$ in the case of $\rho < 0$. In the limiting case when $\rho$ goes down to $-1$, we retrieve the diagonal 
 $\delta(t)=\max(0,2t-1)$, which we have studied earlier in Section \ref{sec:ex-pie-lin}.
 
     \begin{figure}[ht]
   
   \centering
      \begin{subfigure}[b]{.32\linewidth}
       \centering
       \includegraphics[width=50mm]{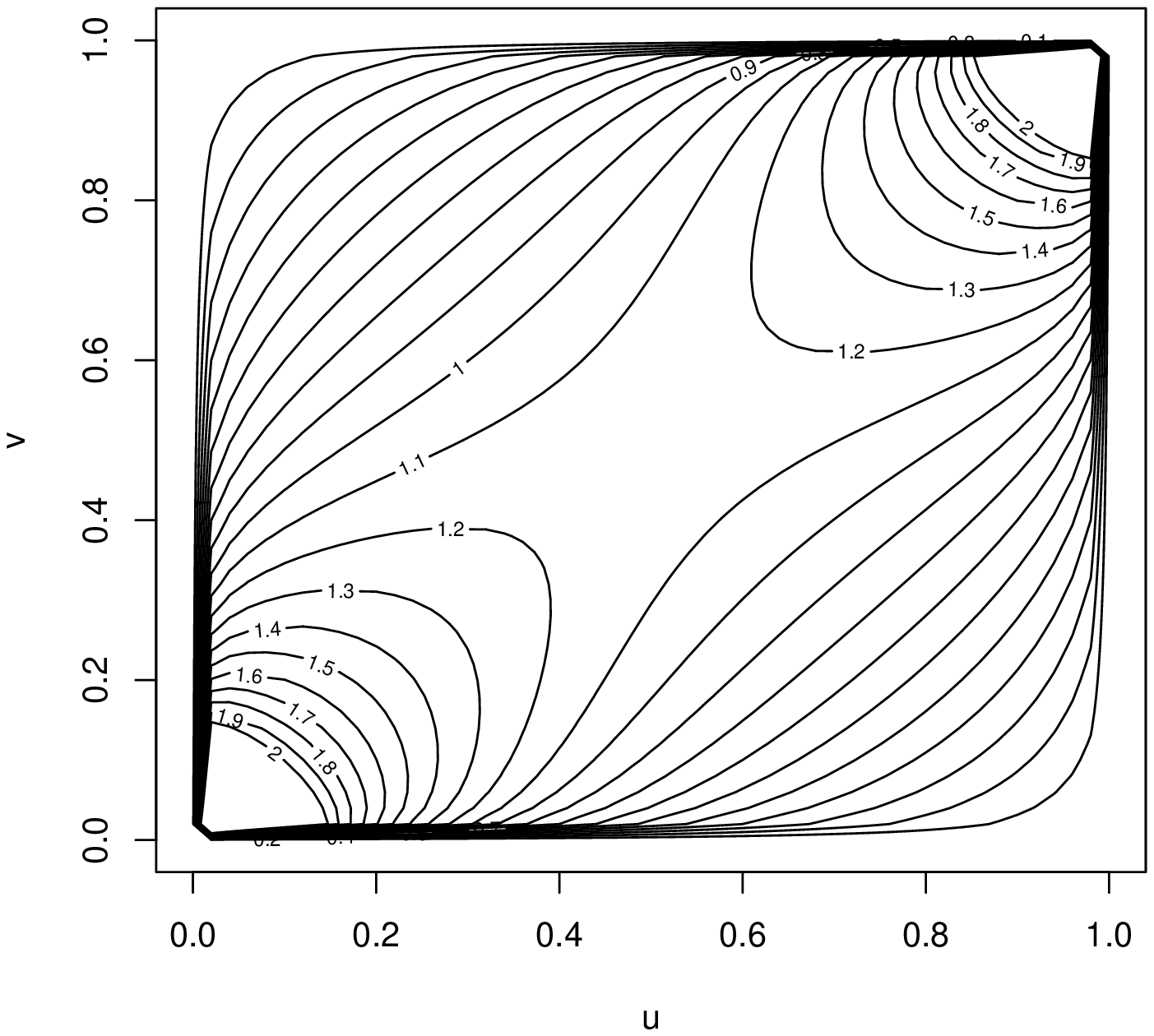}
       \caption{Normal, $\rho=0.5$}
       \label{normal}      
    \end{subfigure}
     \begin{subfigure}[b]{.32\linewidth}
      \centering
      \includegraphics[width=50mm]{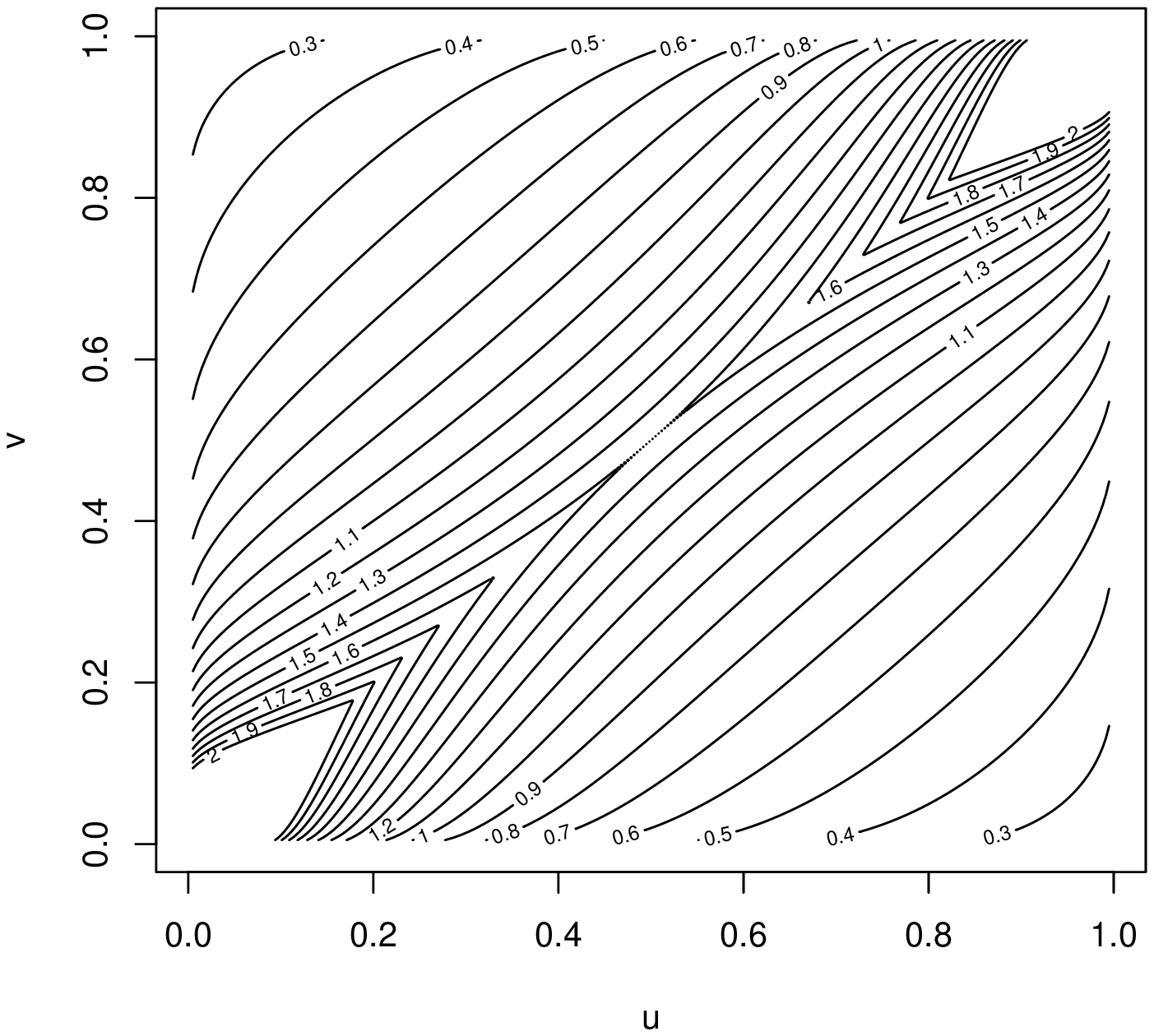}
      \caption{$c_\delta$, $\rho=0.5$}
      \label{normal_max_ent}
   \end{subfigure}
       \begin{subfigure}[b]{.32\linewidth}
      \centering
      \includegraphics[width=50mm]{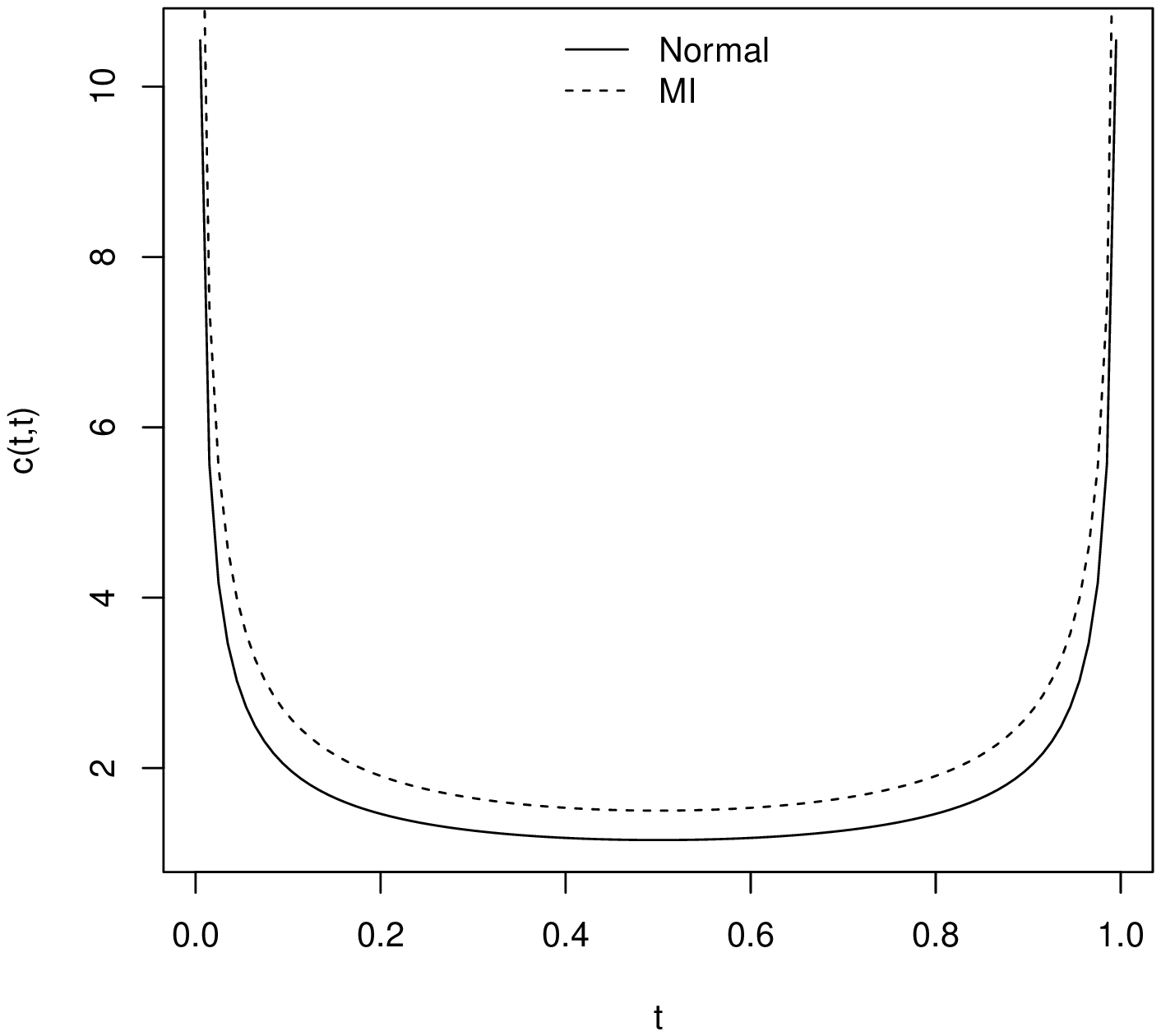}
      \caption{Diagonal cross-section}
      \label{normal_dens_diag}
   \end{subfigure}
   
   \begin{subfigure}[b]{.32\linewidth}
       \centering
       \includegraphics[width=50mm]{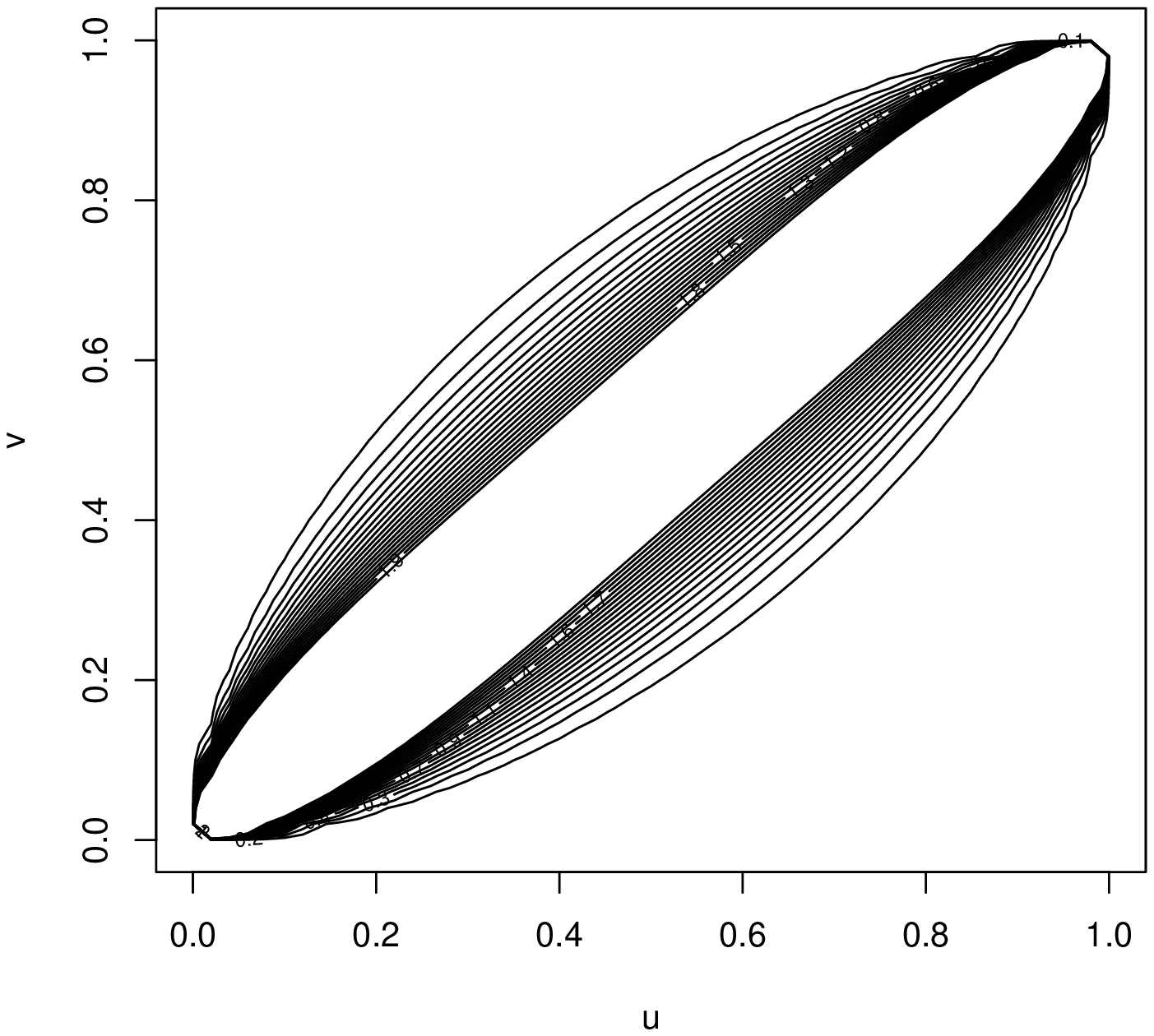}
       \caption{Normal, $\rho=0.95$}
       \label{normal_extr}      
    \end{subfigure}
     \begin{subfigure}[b]{.32\linewidth}
      \centering
      \includegraphics[width=50mm]{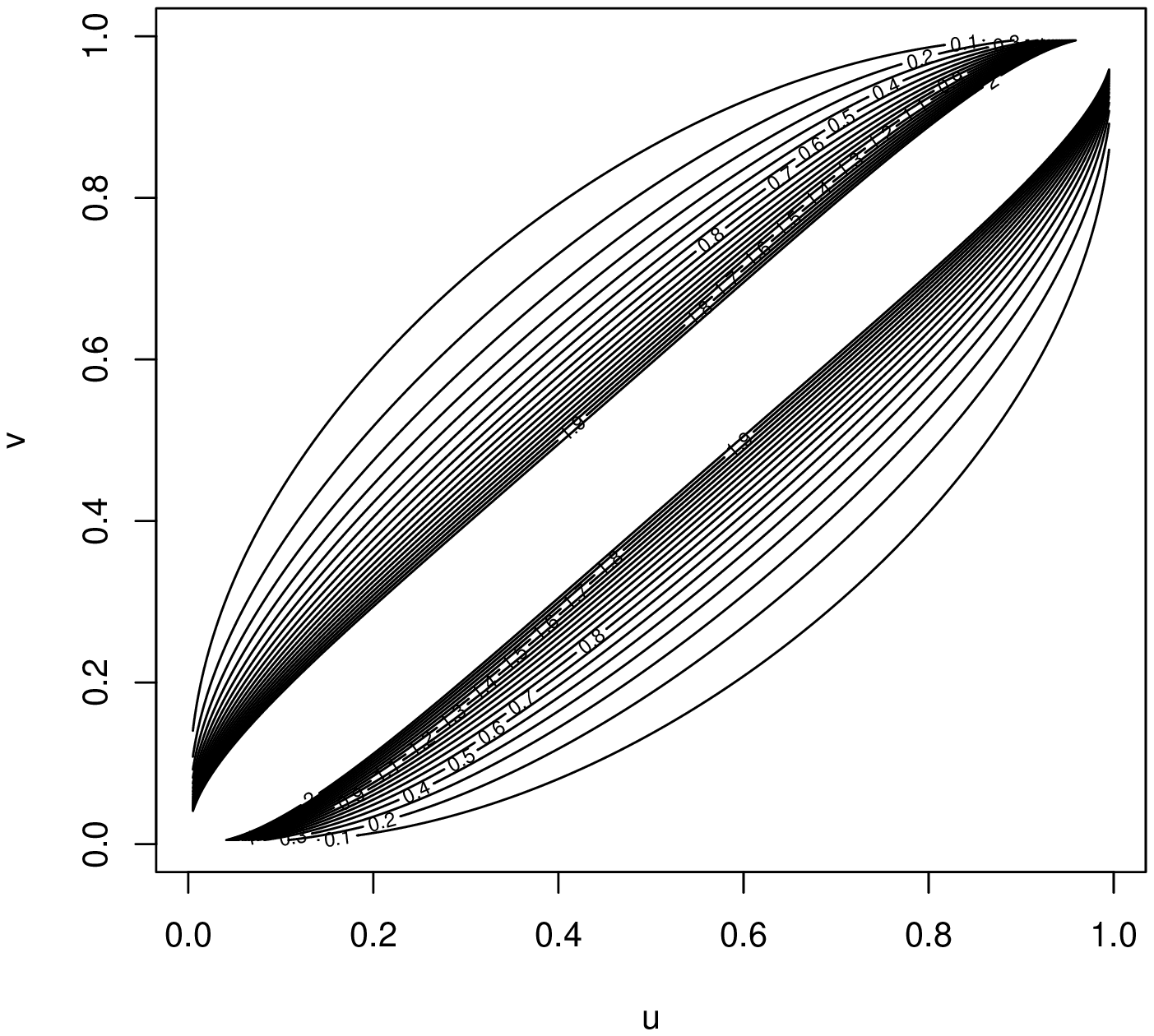}
      \caption{$c_\delta$, $\rho=0.95$}
      \label{normal_extr_max_ent}
   \end{subfigure}
   \begin{subfigure}[b]{.32\linewidth}
      \centering
      \includegraphics[width=50mm]{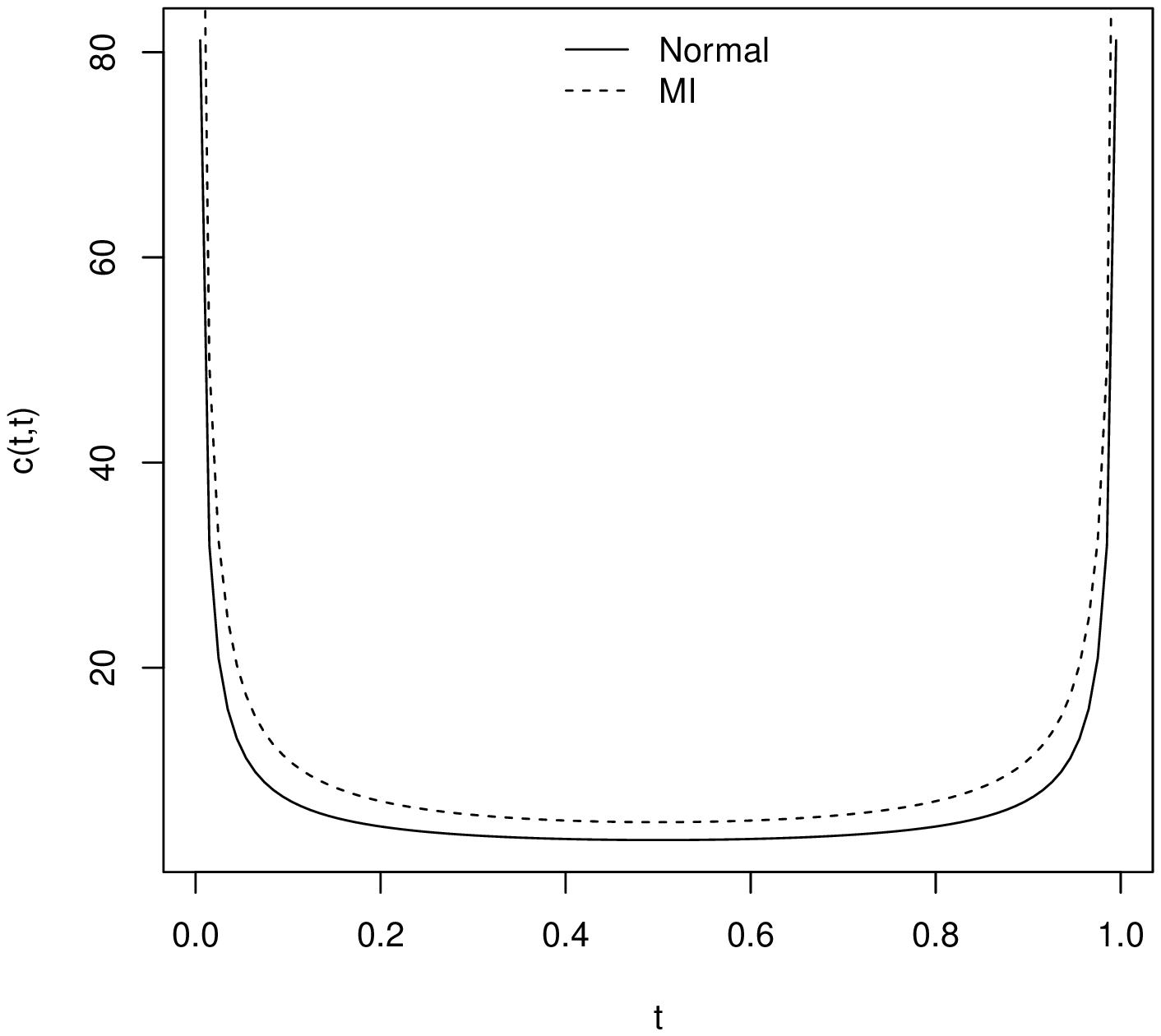}
      \caption{Diagonal cross-section}
      \label{normal_extr_dens_diag}
   \end{subfigure}
      \caption{Isodensity lines and the diagonal cross-section of copulas with diagonal section given by \reff{eq:delta_normal}, with $\rho=0.5$ and
      $\rho=0.95$.}
   \label{normal_iso}  
   \end{figure} 
   
   \begin{figure}[ht]
   
   \centering
   
   \begin{subfigure}[b]{.32\linewidth}
       \centering
       \includegraphics[width=50mm]{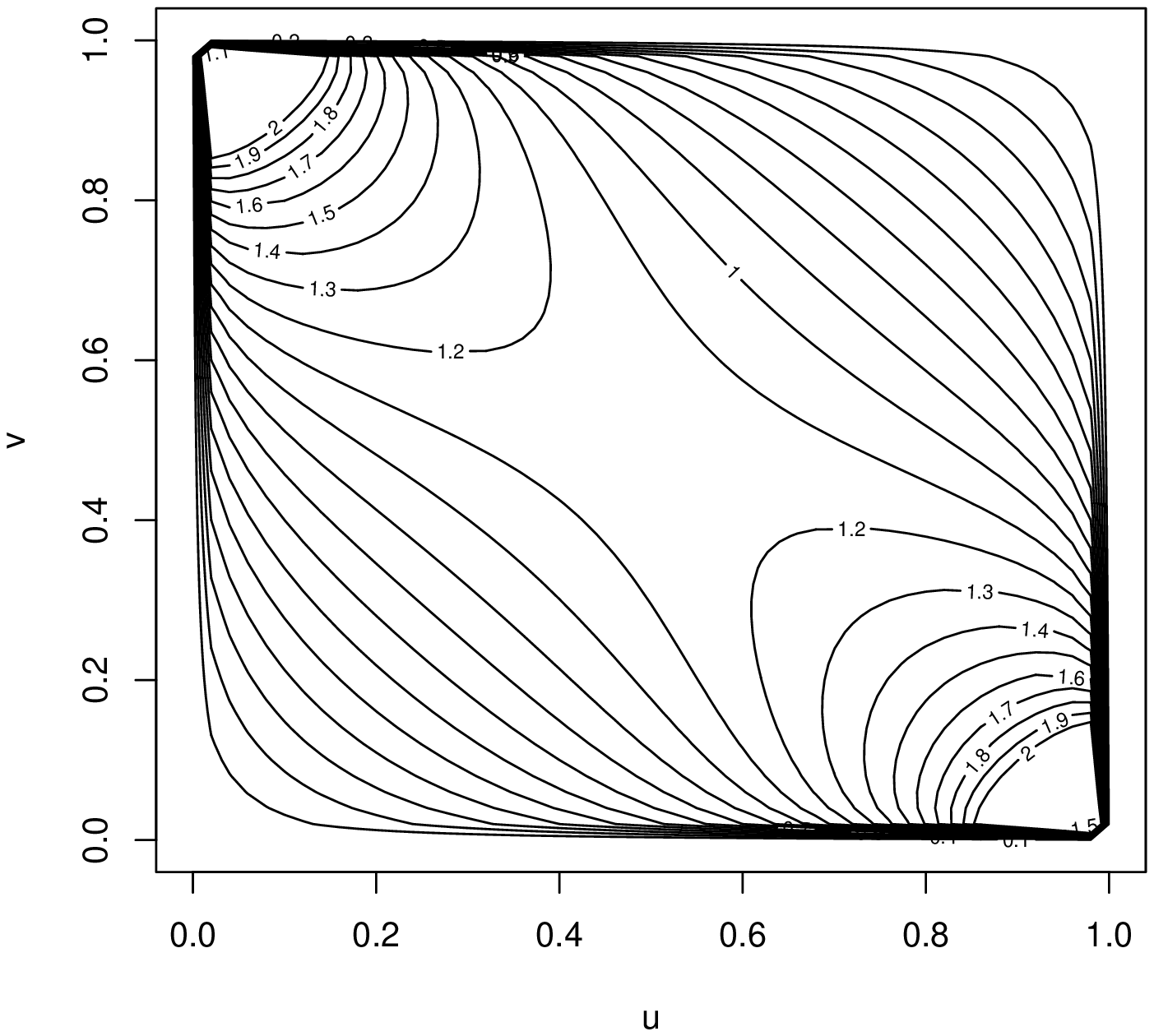}
       \caption{Normal, $\rho=-0.5$}
       \label{normal_neg}      
    \end{subfigure}
     \begin{subfigure}[b]{.32\linewidth}
      \centering
      \includegraphics[width=50mm]{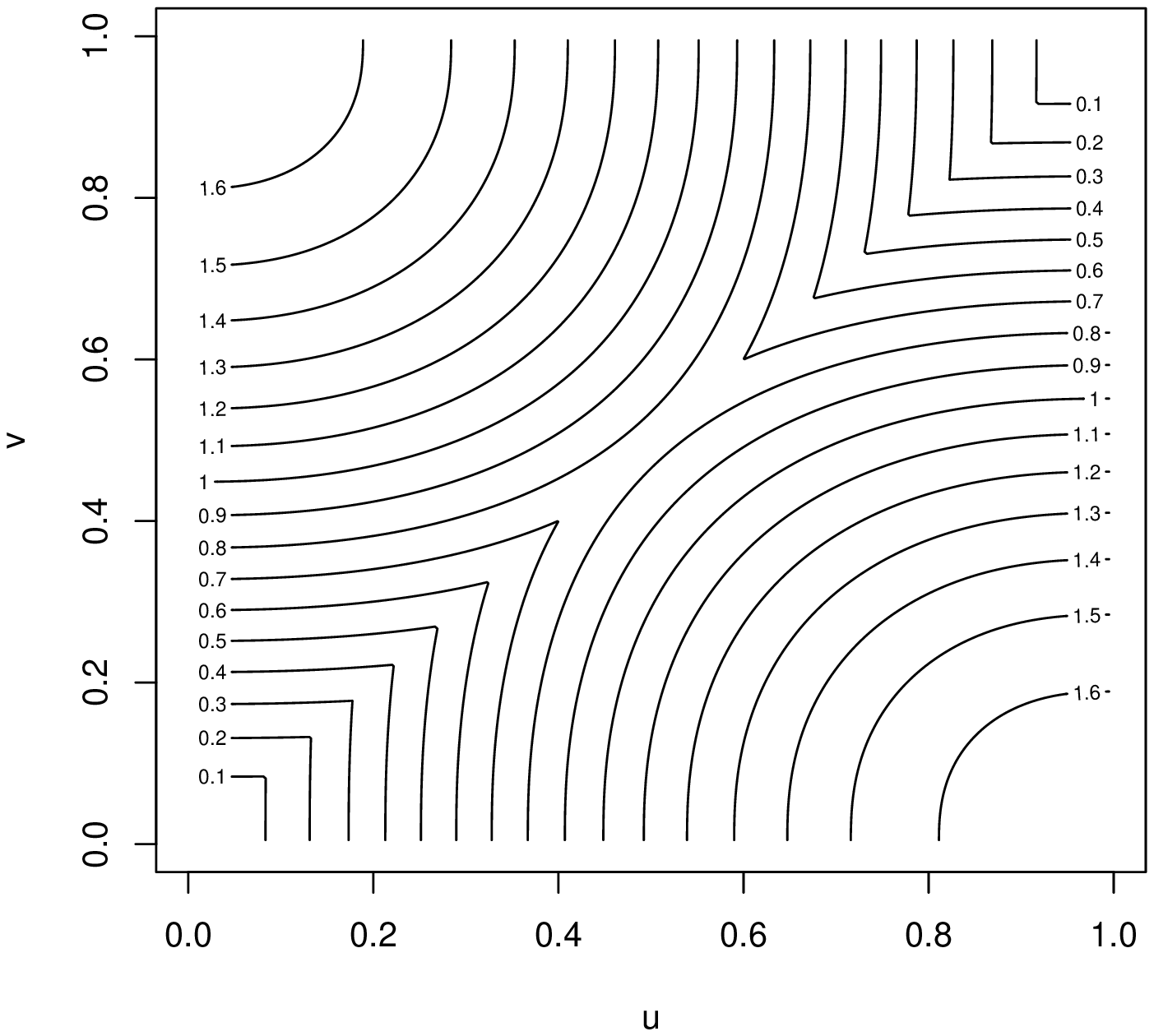}
      \caption{$c_\delta$, $\rho=-0.5$}
      \label{normal_neg_max_ent}
   \end{subfigure}  
    \begin{subfigure}[b]{.32\linewidth}
      \centering
      \includegraphics[width=50mm]{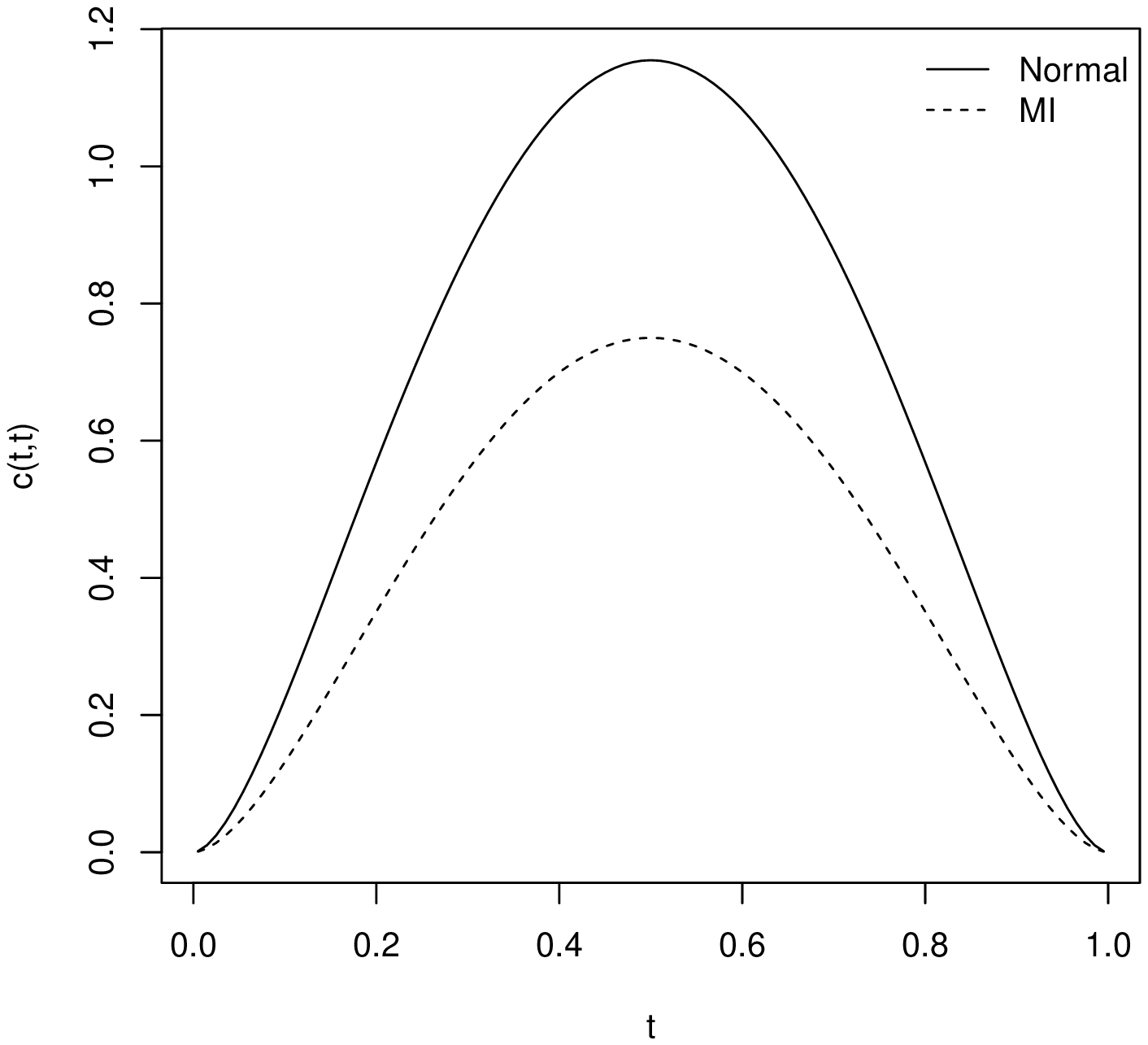}
      \caption{Diagonal cross-section}
      \label{normal_neg_dens_diag}
   \end{subfigure}
   
      \begin{subfigure}[b]{.32\linewidth}
       \centering
       \includegraphics[width=50mm]{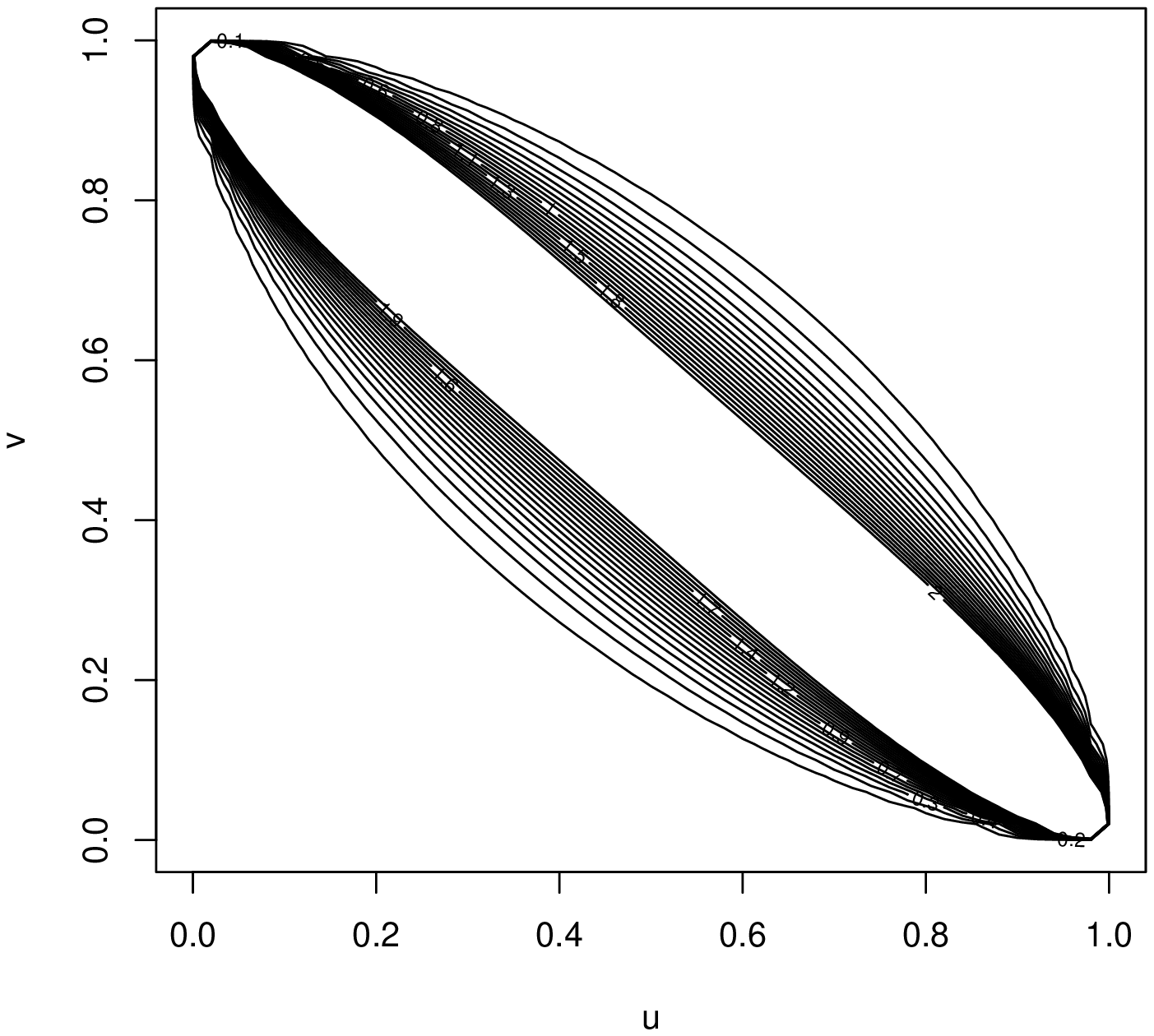}
       \caption{Normal, $\rho=-0.95$}
       \label{normal_extr_neg}      
    \end{subfigure}
     \begin{subfigure}[b]{.32\linewidth}
      \centering
      \includegraphics[width=50mm]{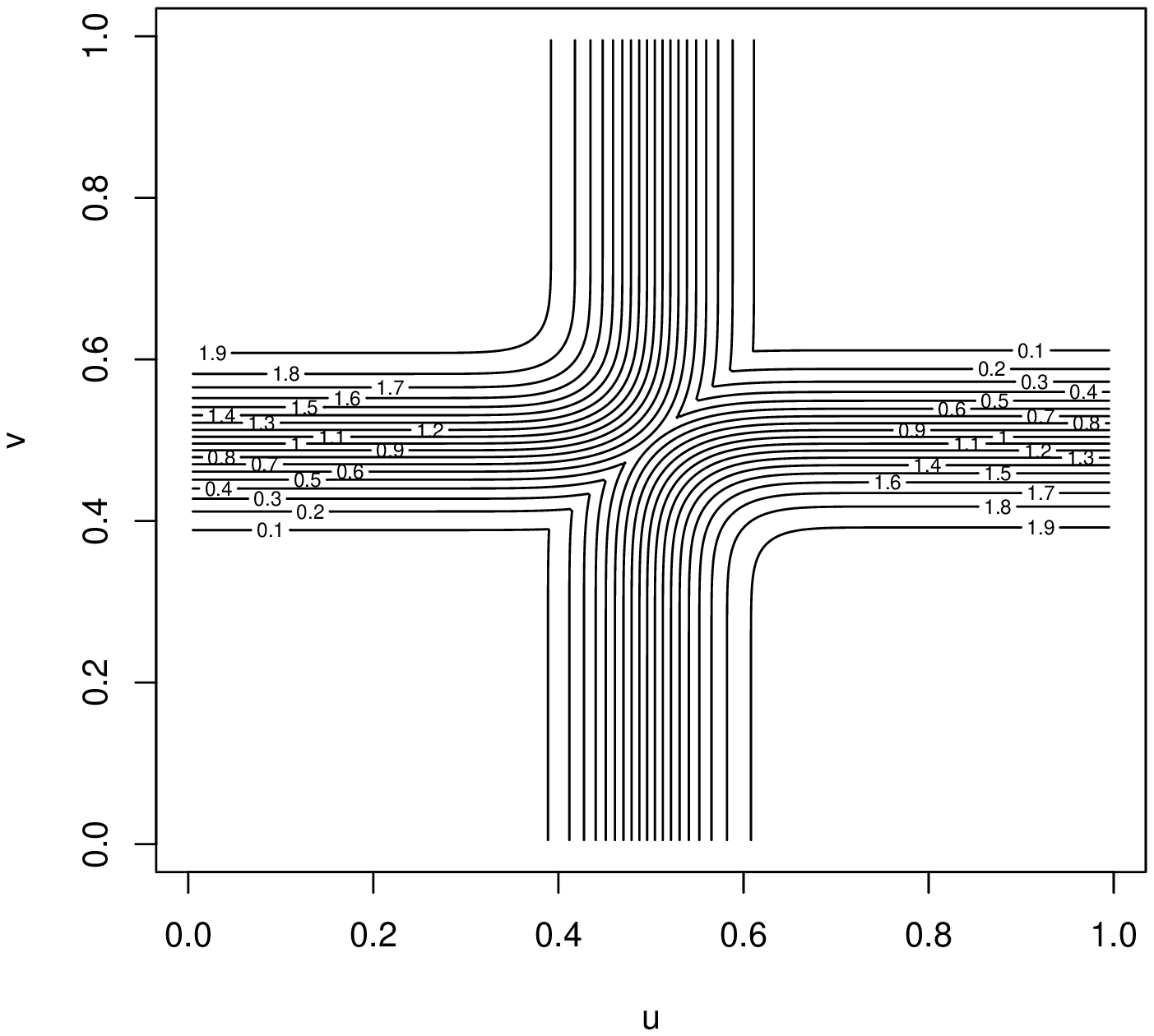}
      \caption{$c_\delta$, $\rho=-0.95$}
      \label{normal_extr_neg_max_ent}
     \end{subfigure}
   \begin{subfigure}[b]{.32\linewidth}
      \centering
      \includegraphics[width=50mm]{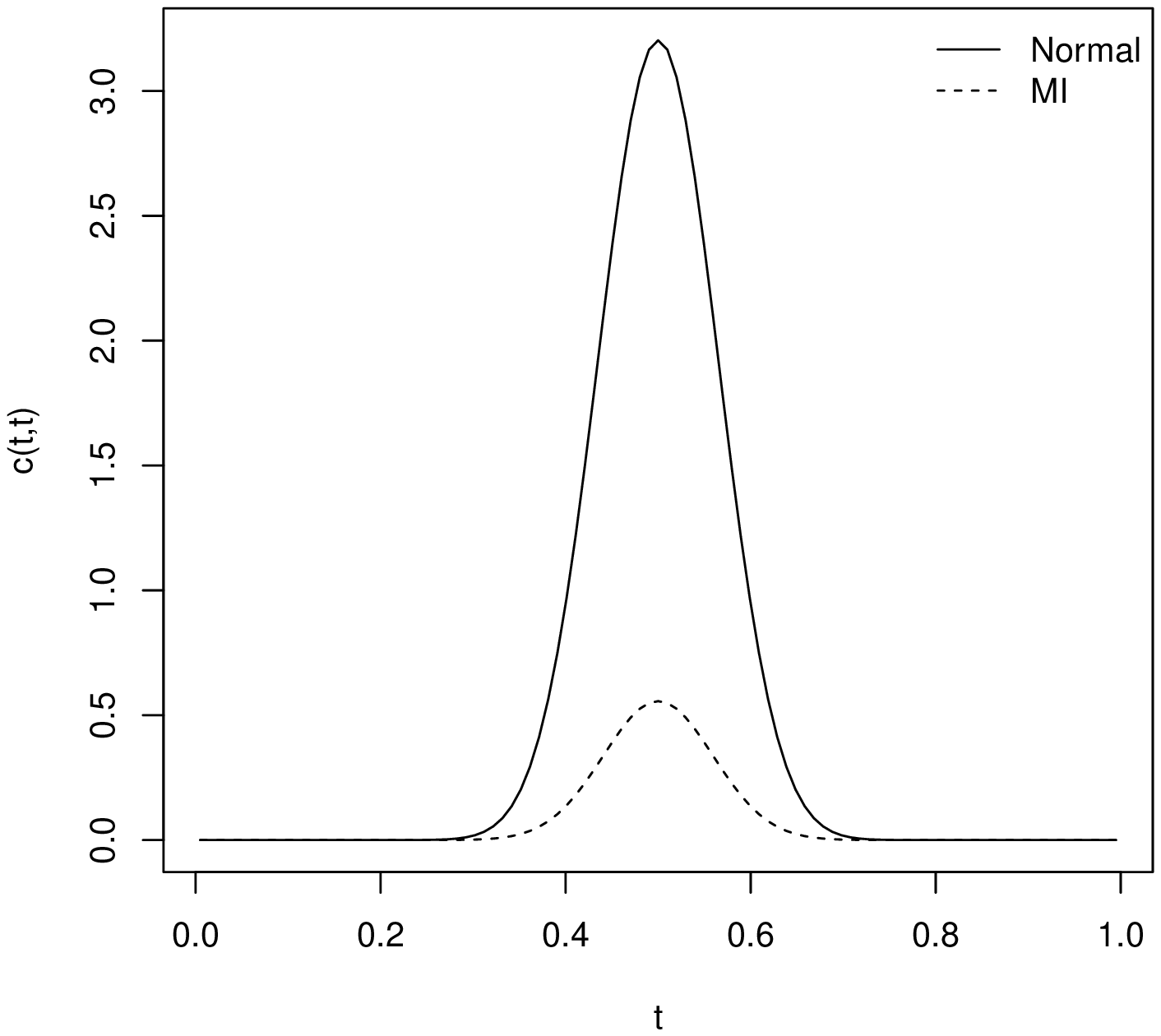}
      \caption{Diagonal cross-section}
      \label{normal_extr_neg_dens_diag}
   \end{subfigure}
   \caption{Isodensity lines and the diagonal cross-section of copulas with diagonal section given by \reff{eq:delta_normal},
    with $\rho=-0.5$ and $\rho=-0.95$}
   \label{normal_iso_neg}  
   \end{figure}

   \begin{figure}[ht]
   
   \centering
   
   \begin{subfigure}[b]{.45\linewidth}
       \centering
       \includegraphics[width=60mm]{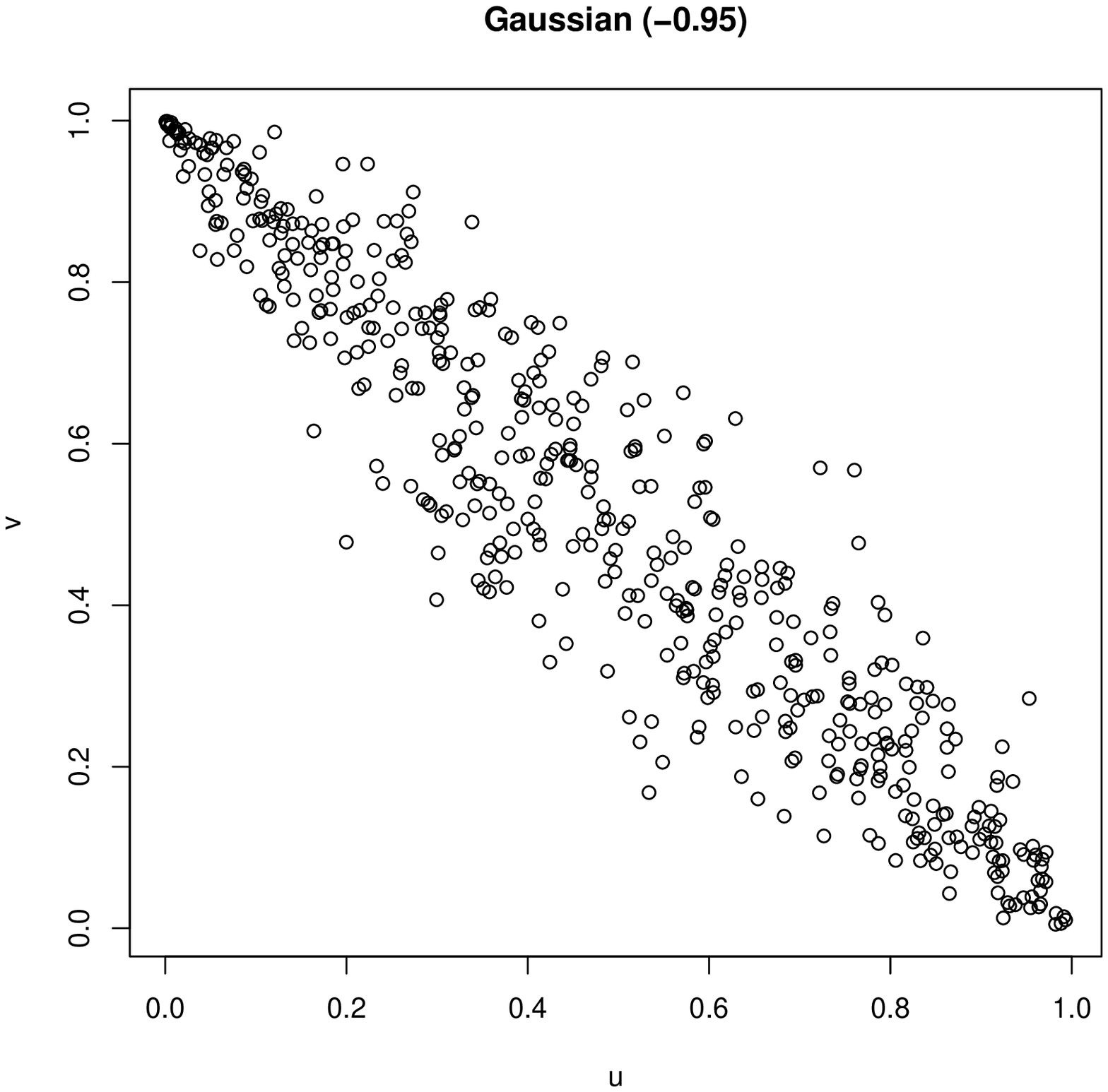}
       \caption{Normal, $\rho=-0.95$}
       \label{Normal_sam}      
    \end{subfigure}
     \begin{subfigure}[b]{.45\linewidth}
      \centering
      \includegraphics[width=60mm]{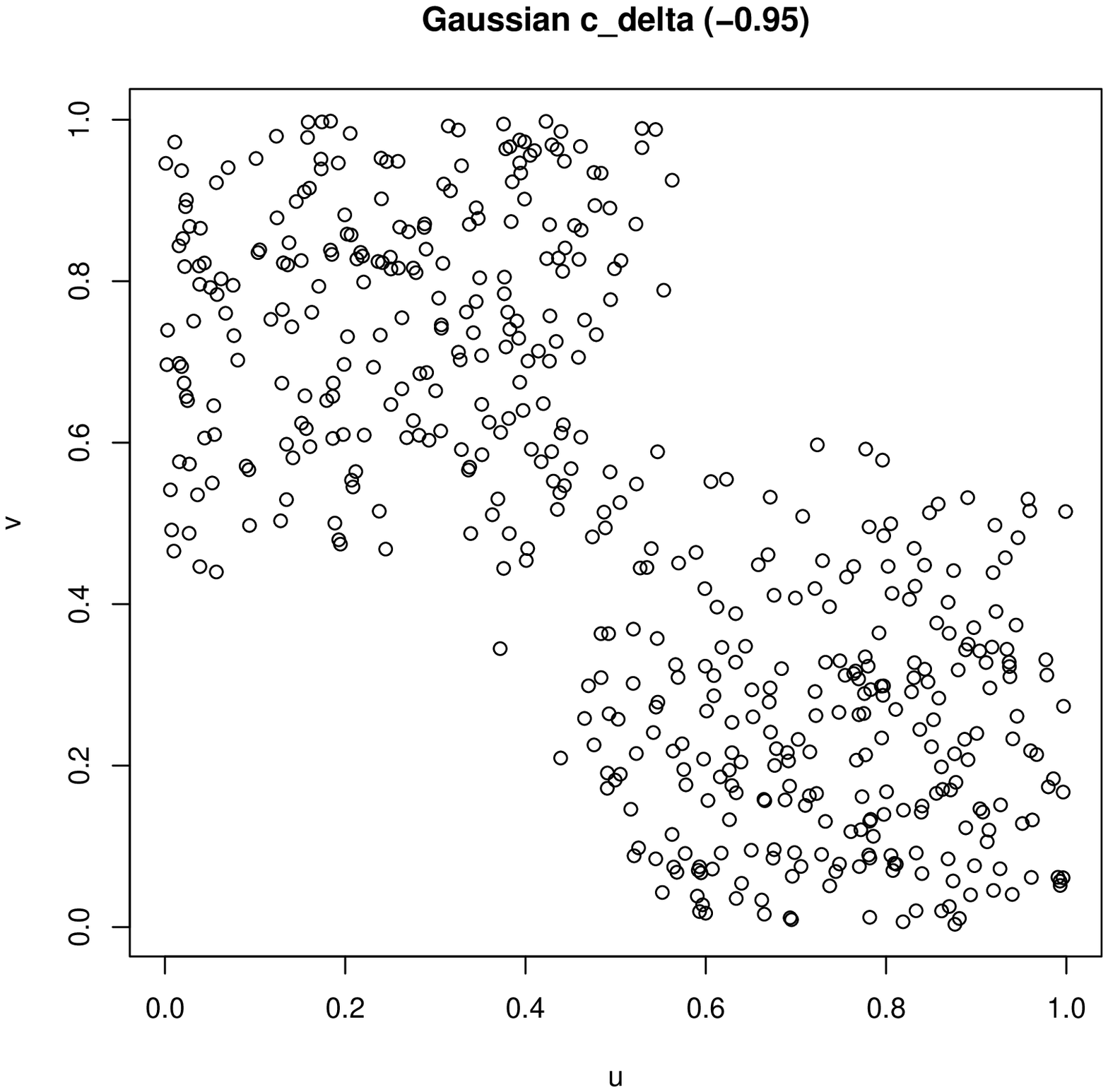}
      \caption{$C_\delta, \rho=-0.95$}
      \label{Normal_max_sam}
   \end{subfigure}
   \caption{Sample of $500$ drawn from the Gaussian copula with $\rho=-0.95$ and from the corresponding $C_\delta$ }
   \label{normal_neg_sample}  
   \end{figure} 

\section{Appendix - Calculation of the entropy of $C_\delta$}
 \label{app:proofKL}
 Let us first introduce some notations. Let $\varepsilon\in (0, 1/2)$.  
Since $x\log(x)\geq  -1/\expp{}$ for $x>0$, we deduce by the monotone
convergence theorem that:
\begin{equation}
   \label{eq:Ice}
\ci(C_\delta)=\lim_{\varepsilon\downarrow 0} \ci_\varepsilon(C_\delta),
\end{equation}
with:
\[
\ci_\varepsilon(C_\delta)=
\int_{[\varepsilon,
  1-\varepsilon]^d} c_\delta(x)\log(c_\delta(x))\, dx.
\]
Using $\delta \leq t$ and that $\delta$ is a non-decresing, $d$-Lipschitz function, we get that
for $t\in I$:
\begin{equation}
\label{eq:h_bounds}
0\leq h(t)\leq \min(t, (d-1)(1-t)) \leq (d-1)\min(t,1-t).
\end{equation}
We set:
\begin{equation}
\label{eq:w_def}
w(t)=a(t)\expp{-F(t)}=\frac{d- \delta'(r)}{d} h^{-1+1/d}(r).
\end{equation}

From the symmetric property of $c_\delta$, we have that

\begin{equation}
   \label{eq:I-Jd}
\ci_\varepsilon(C_\delta)= J_1(\varepsilon)+J_2(\varepsilon)-J_3(\varepsilon),
\end{equation}
 with:
\begin{align*}
   J_1(\varepsilon)
&= d  \int_{[\varepsilon,
  1-\varepsilon]^d} c_\delta(x)  \ind_{\{ \max(x)=x_d \}}
  \left(\sum_{i=1}^{d-1}
  \log\left(w(x_i) \right) 
  \right)\ \, dx,\\
   J_2(\varepsilon)
&= d \int_{[\varepsilon,
  1-\varepsilon]^d} c_\delta(x) \ind_{\{ \max(x)=x_d \}}
  \log\left( \frac{\delta'(x_d)}{d} h^{-1+1/d}(x_d) \right) \ \, dx,\\
   J_3(\varepsilon)
&= d \int_{[\varepsilon,
  1-\varepsilon]^d} c_\delta(x) \ind_{\{ \max(x)=x_d \}}
  \left((d-1)F(x_d) -\sum_{i=1}^{d-1} F(x_i)\right) \, dx.
\end{align*} 
  We introduce
 $A_\varepsilon(r)=\int_\varepsilon^r a(x)  \, dx$. For $J_1(\varepsilon)$, we have:
  \begin{align*}
  J_1(\varepsilon) & = d (d-1) \int_{[\varepsilon,
  1-\varepsilon]} \ind{\{ \max(x)=x_d \}} b(x_d) \prod_{j=1}^{d-1} a(x_j) 
  \log\left( w(x_1) \right)
 \ \, dx\\  
                   & = d (d-1) \int_{[\varepsilon,
  1-\varepsilon]} \left( \int_{[t,1-\varepsilon]}A_\varepsilon^{d-2}(s)b(s) \, ds \right) a(t) 
  \log\left( w(t) \right) \ \, dt.
  \end{align*}
  Notice that using \reff{eq:A-B} and \reff{eq:A^(d-2)}, we have:
  \begin{align*}
   \int_{[t,1-\varepsilon]}A_\varepsilon^{d-2}(s)b(s) \, ds & = \int_{[t,1]} A^{d-2}(s)b(s) \, ds  
            -\int_{[t,1]}  \left(A^{d-2}(s)-A^{d-2}_\varepsilon(s) \right)b(s) \, ds  \\
          & \hspace{1.5cm} -\int_{[1-\varepsilon,1]}A_\varepsilon^{d-2}(s)b(s) \, ds. \\
          & = \frac{h(t)}{(d-1) A(t)} - \int_t^1  \left(A^{d-2}(s)-A^{d-2}_\varepsilon(s) \right)b(s) \, ds \\
          & \hspace{1.5cm} -\int_{[1-\varepsilon,1]}A_\varepsilon^{d-2}(s)b(s) \, ds.
  \end{align*}
  By Fubini's theorem, we get:
  \[
    J_{1}(\varepsilon) = J_{1,1}(\varepsilon) - J_{1,2}(\varepsilon) - J_{1,3}(\varepsilon),
  \]
  with:
  \begin{align*}
  J_{1,1}(\varepsilon) & =  \int_{[\varepsilon,
  1-\varepsilon]}  (d-\delta'(t)) \log\left( w(t) \right) \ \, dt \\
  J_{1,2}(\varepsilon)  & = d (d-1) \left(\int_{[1-\varepsilon,1]}A_\varepsilon^{d-2}(s)b(s) \, ds \right)
                   \int_{[\varepsilon,1-\varepsilon]}  a(t) 
  \log\left( w(t) \right) \ \, dt\\
  J_{1,3}(\varepsilon) & = d (d-1) \int_{[\varepsilon,
  1-\varepsilon]} \left(\int_t^1  \left(A^{d-2}(s)-A^{d-2}_\varepsilon(s) \right) b(s) \, ds \right) a(t) 
  \log\left( w(t) \right) \ \, dt.
  \end{align*}                
To study $J_{1,2}$, we 
first give an upper bound for the term $\int_{[1-\varepsilon,1]}A_\varepsilon^{d-2}(s)a(s)b(s) \, ds$:
\begin{align} \label{eq:A^(d-2)up}
\begin{split}
 \int_{[1-\varepsilon,1]}A_\varepsilon^{d-2}(s)b(s) \, ds & \leq \int_{[1-\varepsilon,1]}A^{d-2}(s)b(s) \, ds \\
          & = \inv{(d-1)}h^{1-1/d}(1-\varepsilon)\expp{-F(1-\varepsilon)} \\
          & \leq (d-1)^{-1/d}\varepsilon^{1-1/d},           
\end{split}
\end{align}
where we used that $A_\varepsilon(s) \leq A(s)$ for $s>\varepsilon$ for the first inequality, \reff{eq:A^(d-2)} for 
the first equality, and \reff{eq:h_bounds} for the last inequality.
Since $ t \log(t) \geq -1/\expp{}$, we have, using \reff{eq:w_def}:
\begin{align*}
 J_{1,2}(\varepsilon) &  \geq -\frac{d (d-1)}{\expp{}}
                         \left(\int_{[1-\varepsilon,1]}A_\varepsilon^{d-2}(s)b(s) \, ds \right) \int_{[\varepsilon,
  1-\varepsilon]}  \expp{F(t)} \ \, dt\\
                      & \geq -\frac{d }{\expp{}}  h^{1-1/d}(1-\varepsilon) \int_{[\varepsilon,
  1-\varepsilon]}\expp{F(t)-F(1-\varepsilon)} \ \, dt\\
                      & \geq -\frac{d}{\expp{}}  ((d-1)\varepsilon)^{1-1/d},\\                      
\end{align*}
 where we used \reff{eq:A^(d-2)} for the second inequality, 
and that $F$ is non-decreasing and \reff{eq:A^(d-2)up} for the third inequality.
On the other hand, we have $t \log(t) \leq t^{\frac{1}{1-1/d}}$, if $t\geq 0$, which gives:
\begin{align*}
 J_{1,2}(\varepsilon) & \leq  d (d-1) \left(\int_{[1-\varepsilon,1]}A_\varepsilon^{d-2}(s)b(s) \, ds \right)
                        \int_{[\varepsilon, 1-\varepsilon]} 
                        \expp{F(t)} \frac{\left(\frac{d-\delta'(t)}{d}\right)^{\frac{1}{1-1/d}}}{h(t)}\ \, dt\\
                      & = d h(1-\varepsilon)^{1-1/d} \int_{[\varepsilon,
  1-\varepsilon]}  \frac{\expp{F(t)-F(1-\varepsilon)}}{h(t)} \ \, dt\\
                      & = d h(1-\varepsilon)^{1-1/d} 
  \left(1-\expp{F(\varepsilon)-F(1-\varepsilon)}\right)\\
                      & \leq d ((d-1)\varepsilon)^{1-1/d},
\end{align*}
where we used \reff{eq:A^(d-2)up} and $t^{\frac{1}{1-1/d}} \leq 1$ for $t \in I$ 
for the first inequality, and that $F$ is non-decreasing for the last. 
This proves that $\lim_{\varepsilon \rightarrow 0} J_{1,2}(\varepsilon) = 0$.
For $J_{1,3}(\varepsilon)$, we first observe that for $s \in [\varepsilon,1-\varepsilon]$ 
we have $A_\varepsilon(s) \leq A(s)$ and thus:
\begin{equation}
  \label{eq:A^(d-2)-A^(d-2)e}
  \left(A^{d-2}(s)-A^{d-2}_\varepsilon(s) \right) = A(\varepsilon)\sum_{i=0}^{d-3}A^i(s)A^{d-3-i}_\varepsilon(s) 
                        \leq (d-2)A(\varepsilon)A^{d-3}(s).                                              
\end{equation}
Using the previous inequality we obtain:
\begin{align*}
 J_{1,3}(\varepsilon) & = d (d-1) \int_{[\varepsilon,
  1-\varepsilon]} \left( \int_t^1  \left(A^{d-2}(s)-A^{d-2}_\varepsilon(s) \right)b(s) \, ds \right) a(t) 
  \log\left( w(t) \right) \ \, dt\\
                      & \geq -\frac{d (d-1)}{\expp{}}\int_{[\varepsilon,
  1-\varepsilon]} \left( \int_t^1 \left(A^{d-2}(s)-A^{d-2}_\varepsilon(s) \right)b(s) \, ds \right) \expp{F(t)}\ \, dt\\
                      & \geq -\frac{d (d-1)(d-2)A(\varepsilon)}{\expp{}}\int_{[\varepsilon,
  1-\varepsilon]} \left( \int_t^1 A^{d-3}(s)b(s) \, ds \right) \expp{F(t)}\ \, dt\\
                      & \geq -\frac{d (d-1)(d-2)A(\varepsilon)}{\expp{}}\int_{[\varepsilon,
  1-\varepsilon]} \frac{\left( \int_t^1 A^{d-2}(s)b(s) \, ds \right)}{A(t)} \expp{F(t)}\ \, dt\\
                      & = -\frac{d (d-2)A(\varepsilon)}{\expp{}}\int_{[\varepsilon,
  1-\varepsilon]} \frac{ h(t)}{A^2(t)} \expp{F(t)}\ \, dt\\  
                      & = -\frac{d (d-2)h^{1/d}(\varepsilon)}{\expp{}}\int_{[\varepsilon,
  1-\varepsilon]} h(t)^{1-2/d} \expp{F(\epsilon)-F(t)}\ \, dt\\  
                      & \geq -\frac{d (d-2) (d-1)^{1-1/d}\varepsilon^{1/d}}{\expp{}},
\end{align*}
where we used $t\log(t) \geq -1/\expp{}$ for the first inequality, \reff{eq:A^(d-2)-A^(d-2)e} for the second, 
\reff{eq:A-B} and \reff{eq:A^(d-2)} in the following equality, and \reff{eq:h_bounds} to conclude.
For an upper bound, we have after noticing that $t \log(t) \leq t^2$ :
\begin{align*}
 J_{1,3}(\varepsilon) & = d (d-1) \int_{[\varepsilon,
  1-\varepsilon]} \left( \int_t^1  \left(A^{d-2}(s)-A^{d-2}_\varepsilon(s) \right)b(s) \, ds \right) a(t) 
  \log\left( w(t) \right) \ \, dt\\
                      & \leq d (d-1) \int_{[\varepsilon,
  1-\varepsilon]} \left( \int_t^1  \left(A^{d-2}(s)-A^{d-2}_\varepsilon(s) \right)b(s) \, ds \right) 
  \expp{F(t)}w^2(t)\ \, dt\\
                      & \leq d (d-1) (d-2) A(\varepsilon) \int_{[\varepsilon,
  1-\varepsilon]} \frac{\left( \int_t^1  A^{d-2}(s)b(s) \, ds \right)}{A(t)}
  \expp{F(t)}h^{-2+2/d}(t)\ \, dt\\
                      & = d (d-2) A(\varepsilon) \int_{[\varepsilon,
  1-\varepsilon]}\frac{\expp{-F(t)}}{h(t)}\ \, dt\\
                      & = d (d-2) h^{1/d}(\varepsilon) (1-\expp{F(\varepsilon)-F(1-\varepsilon)}) \\
                      & \leq d (d-2)(d-1)^{1/d} \varepsilon^{1/d},
\end{align*}
where we used \reff{eq:A^(d-2)-A^(d-2)e} and $0 \leq (d-\delta'(t))/d \leq 1$  for the second inequality; 
\reff{eq:A-B} and \reff{eq:A^(d-2)} in the second equality; and \reff{eq:h_bounds} to conclude. 
The results on the two bounds show that $\lim_{\varepsilon \rightarrow 0} J_{1,3}(\varepsilon)=0$. 
Similarly, for $J_2(\varepsilon)$, we get:
\begin{align*}
  J_2(\varepsilon) & = \int_{[\varepsilon,1-\varepsilon]^d}\ind{\{ \max(x)=x_d \}}b(x_d)\prod_{j=1}^{d-1} a(x_j)  
  \log\left( \frac{\delta'(x_d)}{d} h^{-1+1/d}(x_d)  \right) \ \, dx\\  
                   & = d \int_{[\varepsilon,
  1-\varepsilon]} A_\varepsilon^{d-1}(t)b(t)
  \log\left( \frac{\delta'(t)}{d} h^{-1+1/d} (t)  \right) \ \, dt \\
                & = d\int_{[\varepsilon, 1-\varepsilon]} A^{d-1}(t)b(t) 
  \log\left( \frac{\delta'(t)}{d} h^{-1+1/d}(t)   \right) \ \, dt \\
                & \hspace{1.5cm}- d \int_{[\varepsilon, 1-\varepsilon]} \left(A^{d-1}(t)-A^{d-1}_\varepsilon(t)\right)b(t) 
  \log\left( \frac{\delta'(t)}{d} h^{-1+1/d} (t)  \right) \ \, dt \\
                & = J_{2,1}(\varepsilon)-J_{2,2}(\varepsilon)
 \end{align*}
 with $J_{2,1}(\varepsilon)$ and $J_{2,2}(\varepsilon)$ given by, using \reff{eq:A^(d-1)}:
 
 \begin{align*}
  J_{2,1}(\varepsilon) & = d\int_{[\varepsilon, 1-\varepsilon]} A^{d-1}(t)b(t) 
  \log\left( \frac{\delta'(t)}{d} h^{-1+1/d}(t)   \right) \ \, dt \\
  J_{2,2}(\varepsilon) & = d \int_{[\varepsilon, 1-\varepsilon]} \left(A^{d-1}(t)-A^{d-1}_\varepsilon(t)\right)b(t) 
  \log\left( \frac{\delta'(t)}{d} h^{-1+1/d} (t) \right) \ \, dt.
 \end{align*}
 By \reff{eq:A^(d-1)}, we have:
 \begin{equation}
  J_{2,1}(\varepsilon)=
  \int_{[\varepsilon, 1-\varepsilon]} \delta'(t)
  \log\left( \frac{\delta'(t)}{d} h^{-1+1/d}(t)   \right) \ \, dt.
 \end{equation}

 Similarly to $J_{1,3}(\varepsilon)$ we can show that $\lim_{\varepsilon \rightarrow 0} J_{2,2}(\varepsilon) = 0$.

Adding up $J_1(\varepsilon)$ and $J_2(\varepsilon)$ gives

\begin{align*}
 J_1(\varepsilon)+J_2(\varepsilon) = \cj_{\varepsilon}(\delta) + J_4(\varepsilon) 
 -d \log(d) (1-2\varepsilon) - J_{1,2}(\varepsilon)-J_{1,3}(\varepsilon)-J_{2,2}(\varepsilon)
\end{align*}
with
\[
\cj_\varepsilon(\delta)= (d-1) 
\int_{\varepsilon}^{1-\varepsilon} \val{\log
\left(h(t)\right) } \, dt ,
\]
\[
J_4(\varepsilon)=\int_{\varepsilon}^{1-\varepsilon} \left(d- \delta'(t)\right) \log \left(
  d -\delta'(t)\right)  dt 
+\int_{\varepsilon}^{1-\varepsilon} \delta'(t) \log \left(\delta'(t)\right)  dt.
\]
Notice that $\cj_\varepsilon(\delta)$ is non-decreasing in
$\varepsilon>0$ and that:
\[
\cj(\delta)=\lim_{\varepsilon\rightarrow 0}  \cj_\varepsilon(\delta).
\]
Since $\delta'(t)\in [0,d]$, we deduce that $(d-\delta')\log(d-\delta')$
and $\delta'\log(\delta')$ are bounded on $I$ from above by $d\log(d)$ and from
below by $-1/\expp{}$ and therefore integrable on $I$. This implies :
\[
\lim_{\varepsilon \rightarrow 0} J_4(\varepsilon)=
 \ci_1(\delta')+\ci_1(d-\delta').
\]

As for $J_3(\varepsilon)$, we have by integration by parts:

\begin{align*}
 J_3(\varepsilon)  & = d \int_{[\varepsilon,
  1-\varepsilon]^d} \ind_{\{ \max(x)=x_d \}} b(x_d) \prod_{i=i}^{d-1} a(x_i)  
                    \left((d-1)F(x_d) -\sum_{i=1}^{d-1} F(x_i)\right) \, dx \\
                  & = d (d-1) \int_{[\varepsilon,
  1-\varepsilon]} A_\varepsilon^{d-1}(t)b(t)F(t) \, dt \\
                  & \hspace{1.5cm} - d  (d-1) \int_{[\varepsilon,
  1-\varepsilon]}   A_\varepsilon^{d-2}(t)b(t) \left(\int_\varepsilon^t a(s)  F(s) \right) \, dt \\
                  & = d (d-1) \int_{[\varepsilon,
  1-\varepsilon]} A_\varepsilon^{d-1}(t)b(t)F(t) \, dt \\
                  & \hspace{1.5cm} - d  (d-1) \int_{[\varepsilon,
  1-\varepsilon]}   A_\varepsilon^{d-2}(t)b(t) \left(A_\varepsilon(t)F(t) - \frac{d-1}{d} 
                  \int_\varepsilon^t \frac{A_\varepsilon(s)}{h(s)} \, ds \right) \, dt \\
                  & = (d-1)^2\int_{[\varepsilon,1-\varepsilon]} \left( \int_t^{1-\varepsilon} A_\varepsilon^{d-2}(s)b(s)
                  \right) \frac{A_\varepsilon(t)}{h(t)} \, dt.  \\
\end{align*}
By the monotone convergence theorem, \reff{eq:A-B} and \reff{eq:A^(d-2)} we have:
\begin{align*}
 \lim_{\varepsilon \rightarrow 0} J_3(\varepsilon) &= (d-1)^2\int_{I} \left( \int_t^1 A^{d-2}(s)b(s)
                  \right) \frac{A(t)}{t-\delta(t)} \, dt  \\
                                                   & = d-1.
\end{align*}

Summing up all the terms and taking the limit $\varepsilon=0$ give :
\begin{align*}
 \ci(C_\delta) & = (d-1)\int_I \val{\log(t-\delta(t))} \, dt + \ci_1(\delta') + \ci_1(d-\delta') - d\log(d) - (d-1) \\
               & = (d-1)\cj(\delta) + \cg(\delta).
\end{align*}

 \bibliographystyle{abbrv}
\bibliography{Max_entr_symm_biblio}

\end{document}